\documentclass[reqno,twoside,12pt,a4paper]{amsart}
\usepackage{layout}
\usepackage{latexsym}
\usepackage{amsmath}
\usepackage{amssymb}
\usepackage{ascmac}
\usepackage{mathrsfs}

\usepackage[dvipdfmx]{graphicx}
\usepackage[dvipdfmx]{color}

\usepackage{tikz}
\usetikzlibrary{arrows, calc, intersections}
\usepackage{amsmath}
\usepackage{cases} 
\usepackage{caption}
\usepackage{cite}
\renewcommand{\d}{\, {\mathrm d}}
\newcommand{\dx}{\, {\mathrm d} x}
\newcommand{\dg}{\, {\mathrm d} \Gamma}
\newcommand{\dS}{\, {\mathrm d} S}

\newcommand{\dt}{\, {\mathrm d} t}

\numberwithin{equation}{section}

\newtheorem{theorem}{Theorem}[section]

\newtheorem{lemma}[theorem]{Lemma}
\newtheorem{proposition}[theorem]{Proposition}
\newtheorem{definition}[theorem]{Definition}

\renewenvironment{proof}{\noindent {\bf Proof.}}{\hfill $\Box$}

\allowdisplaybreaks[4]

\title[Asymptotic analysis of transmission problems]
{Asymptotic analysis of transmission problems with parameter-dependent Robin conditions}

\author[T.\ Fukao]{Takeshi Fukao}
\address{Takeshi Fukao: Faculty of Advanced Science and Technology, Ryukoku University, 
1-5 Yokotani, Seta Oe-cho, Otsu-shi, Shiga 520-2194, Japan}
\email{fukao@math.ryukoku.ac.jp}
\thanks{}

\dedicatory{}
\begin{document}

\thispagestyle{empty}

\begin{abstract} 
We study a transmission problem of {N}eumann--{R}obin type involving a parameter $\alpha$ and perform an asymptotic analysis with respect to $\alpha$.
The limits $\alpha \to 0$ and $\alpha \to +\infty$ correspond respectively to complete decoupling and full unification of the problem, and we obtain rates of convergence for both regimes.
Biologically, the model describes two cells connected by a gap junction with permeability $\alpha$: 
the case $\alpha \to 0$ corresponds to a situation where the gap junction is closed, 
leaving only tight junctions between the cells so that no substance exchange occurs, 
while $\alpha \to +\infty$ corresponds to a situation that can be interpreted as the cells forming a single structure. 
We also clarify the relationship between the asymptotic analysis with respect to the parameter $\alpha$ 
and the asymptotics of the system in connection with the convergence of convex functionals known as {M}osco convergence. 
Finally, we consider time-dependent permeability and analyze the case where $\alpha$ blows up in finite time.
Under suitable regularity assumptions, we show that the solution can be extended beyond the blow-up time, remaining in the single structure regime. 
\smallskip

\noindent {\sc Key words:} transmission problem, {N}eumann--{R}obin, parameter-dependent {R}obin,\\ mixed boundary condition, gap junction.
\smallskip

\noindent {\sc Mathematics Subject Classification 2020:} 
34G25, 
35K20, 
35M13, 
35Q92. 
\end{abstract}
\maketitle

\section{Introduction}
\label{intro}
\setcounter{equation}{0} 

A transmission problem (see, e.g., \cite{Bar98, HW21, Lio71, Nec67, Sch00, Sch00b, Sch00c, Wlo87} for well-posedness results and related topics) is a type of boundary value problem for partial differential equations. 
It involves two domains that share a common surface, 
where multiple boundary conditions are imposed on this surface by incorporating the unknown functions from each domain into the boundary conditions of the other. 
At first glance, this may appear to be overdetermined. 
However, since the so-called given data that determine the mutual boundary conditions are themselves unknown functions, 
the number of conditions is actually well balanced.
In this paper, we impose the so-called {R}obin type conditions as transmission conditions treated in \cite{Att84, Bre22}. 
We focus on the coefficients appearing in the transmission conditions and perform asymptotic analysis. 
\smallskip

One of the research motivations stems from an interesting phenomenon that appears in biology, 
known as \emph{gap junctions}. 
There are characteristic junctions between cells that permit the direct diffusion of ions and small molecules (see \cite{GP09}). 
Cells approach each other at distances of only $2$ to $3$ nanometers ({\rm nm}) and are connected by connexin proteins, 
which assemble into \emph{gap junction} channels. 
Since cells are on the micrometer scale ($1\,\mu{\rm m} = 10^3\,{\rm nm}$), 
the $2$ to $3$ nanometer separation means they can be considered as essentially in contact. 
While biological studies of these junctions have been conducted from a wide range of perspectives, 
our interest lies in the mathematical biology setting, 
where such connections can be modeled as diffusion processes governed by concentration gradients of target substances.
Typical intracellular ions and small molecules are present at millimolar concentrations ($1$\,{\rm mM} $\sim$ $10^{20}$\,molecules per litre). 
Converting this to the scale of a cubic micrometer (the scale of a cellular organelle), 
we obtain (from $1\,{\rm L}=10^3\,{\rm cm}^3=10^{15}\,\mu{\rm m}^3$)
\begin{equation*}
6.022 \times 10^{20}\,{\rm molecules }/{\rm L} = 6.022 \times 10^{5}\,{\rm molecules }/ \mu{\rm m}^3.
\end{equation*}
Thus, even within a single cubic micrometer, 
there are about $10^{5}$ to $10^{6}$ molecules, making the continuum approximation through averaging highly reasonable. 
\smallskip

In particular, the {R}obin type transmission condition provides a natural mathematical representation of a semi-permeable surface.
The flux across the surface is proportional to the difference in concentration between the two domains. 
Moreover, the magnitude of this relation is characterized by the permeability constant.
From this perspective, it is natural to classify related biological situations by the size of the permeability parameter $\alpha$.
For instance, tight junctions seal neighboring cells together and effectively prevent the passage of ions and small molecules. 
Even when the permeability is extremely small or zero and \emph{gap junctions} are closed, tight junctions prevent cells from becoming disconnected from one another. 
Mathematically, this corresponds to the limit $\alpha \to 0$.
At the opposite extreme, the limit $\alpha \to +\infty$ mathematically corresponds to a situation that can be interpreted as the cells forming a single structure, 
though this may not fully represent complete cell fusion where two cells merge into a single continuous cytoplasm. 
\emph{Gap junctions} then occupy the intermediate regime $0 < \alpha < +\infty$, where the permeability is finite and allows controlled exchange between cells.
Biologically, these three processes are distinct and independent phenomena, but in the mathematical framework of transmission problems they can be unified through the asymptotic behavior of a single parameter $\alpha$.
This perspective highlights a novelty of the present work. 
We attempt to bring together biologically independent mechanisms under a common mathematical description by analyzing how the transmission conditions depend on $\alpha$.
The finite case, which is the most relevant for modeling \emph{gap junctions}, has been investigated in detail in the mathematical literature (see, e.g., \cite{Bre22, Suk23}).
\smallskip

Let $T>0$ be the terminal time and $\Omega \subset \mathbb{R}^d$ be a bounded domain occupied by the material, 
with its boundary $\Gamma:=\partial \Omega$, $d \in \mathbb{N}$ with $d \ge 2$. 
Let us assume that the domain $\Omega$ is decomposed into two subdomains $\Omega_1$ and $\Omega_2$ with $|\Omega_i| \ne 0$ for $i=1,2$, where 
$|\Omega_i|$ denotes the measure of $\Omega_i$. The surface between $\Omega_1$ and $\Omega_2$ is named by 
$S$ which is given and fixed, that is, $\Omega:=\Omega_1 \cup S \cup \Omega_2$. In this paper, 
we consider the following two cases: 
\smallskip

\begin{itemize}
\item[Case 1:]
Consider the case where $\Omega_1$ and $\Omega_2$ are subdomains resembling cells found in biology, 
which are in contact at the open surface $S$ and together constitute a domain $\Omega$. 
In this configuration, the boundary of domain $\Omega$ consists of boundaries $\Gamma_1$, $\Gamma_2$, and the boundary of $S$, namely $\Gamma_i$ is defined as $\Gamma_i=\partial \Omega_i \setminus \overline{S}$. 
Here, $\partial \Omega_i=\Gamma_i \cup \partial S \cup S=\Gamma_i \cup \partial \Gamma_i \cup S$. 
However, since triple junctions occur at the contact points, 
it should be noted that even when the subdomains $\Omega_i$ are sufficiently smooth, 
the domain $\Omega$ may possess corners (see, the first figure of {\sc Figure} \ref{fig:case1}). 
Conversely, even when $\Omega$ is sufficiently smooth, 
the subdomains $\Omega_i$ may become non-smooth (i.e., possess corners) 
when separated by the surface $S$ that divides smooth domain into two parts (see, the second figure of {\sc Figure} \ref{fig:case1}). 
\begin{figure}[h]
{\setlength{\unitlength}{1cm} 
\begin{tikzpicture}[scale=1]
\draw[thick,rounded corners=15pt] (-2,0) -- (-2,2) -- (0,2) -- (0,1.6);
\draw[thick,rounded corners=15pt] (-2,0) -- (-2,-2) -- (0,-2) -- (0,-1.6);
\put(-1,0){\makebox(0,0){$\Omega_1$}};
\put(-2.3,0){\makebox(0,0){$\Gamma_1$}};
\draw[thick,rounded corners=15pt] (2,0) -- (2,2) -- (0,2) -- (0,1.6);
\draw[thick,rounded corners=15pt] (2,0) -- (2,-2) -- (0,-2) -- (0,-1.6);
\put(1,0){\makebox(0,0){$\Omega_2$}};
\put(2.3,0){\makebox(0,0){$\Gamma_2$}};
\draw[thin] (0,-1.6) -- (0,1.6); 
\put(0.3,0){\makebox(0,0){$S$}};
\draw[->,>=stealth] (2,1)--(2.5,1);
\put(2.3,1.3){\makebox(0,0){$\boldsymbol{n}$}};
\draw[->,>=stealth] (0,1)--(0.5,1);
\put(0.3,1.3){\makebox(0,0){$\boldsymbol{\nu}$}};
\draw[->,>=stealth] (0,1)--(-0.5,1);
\put(-0.4,1.3){\makebox(0,0){$-\boldsymbol{\nu}$}};
%
%
\draw[thick,rounded corners=15pt] (6,-2) -- (4,-2) -- (4,2) -- (6,2);
\draw[thick,rounded corners=15pt] (6,2) -- (8,2) -- (8,-2) -- (6,-2);
\put(5,0){\makebox(0,0){$\Omega_1$}};
\put(3.7,0){\makebox(0,0){$\Gamma_1$}};
\put(7,0){\makebox(0,0){$\Omega_2$}};
\put(8.3,0){\makebox(0,0){$\Gamma_2$}};
\draw[thin] (6,-2) -- (6,2); 
\put(6.3,0){\makebox(0,0){$S$}};
\draw[->,>=stealth] (8,1)--(8.5,1);
\put(8.3,1.3){\makebox(0,0){$\boldsymbol{n}$}};
\draw[->,>=stealth] (6,1)--(6.5,1);
\put(6.3,1.3){\makebox(0,0){$\boldsymbol{\nu}$}};
\draw[->,>=stealth] (6,1)--(5.5,1);
\put(5.6,1.3){\makebox(0,0){$-\boldsymbol{\nu}$}};
\end{tikzpicture}
\caption{Two subdomains $\Omega_1$ and $\Omega_2$ in contact at surface $S$. Triple junctions occur at the contact points.}
\label{fig:case1}
}
\end{figure}
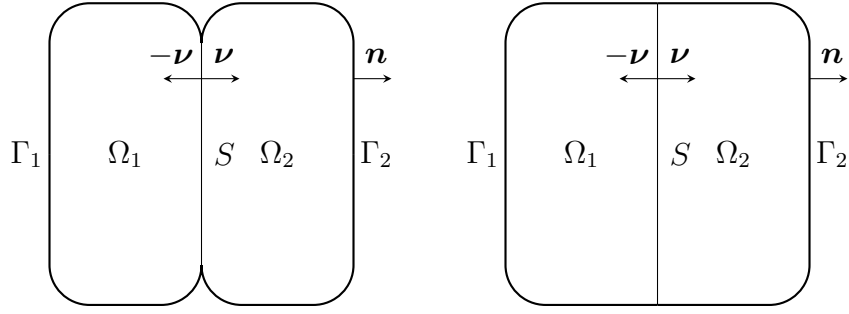

\item[Case 2:]
Next, let us consider the case where $\Omega_1$ is surrounded by $\Omega_2$. 
This situation can be understood more intuitively by considering a cell $\Omega_1$ enclosed by some thin membrane $\Omega_2$. 
If we denote the boundary $\partial \Omega_1$ of $\Omega_1$ as $S$, 
then the surface between $\Omega_1$ and $\Omega_2$ is $S$, 
and the boundary of $\Omega$ is the outermost part $\partial \Omega_2 \setminus S$. 
We denote this as $\Gamma_2$. In this case, $\Gamma_1=\emptyset$. 
For this configuration, the smoothness of each domain can be reasonably assumed, and the regularity issues are expected to be significantly simpler than those in the previous case
(see, {\sc Figure} \ref{fig:case2}). 

\begin{figure}[h]
{\setlength{\unitlength}{1cm} 
\begin{tikzpicture}[scale=1]
\draw[thick,rounded corners=15pt] (0,-2) -- (-2,-2) -- (-2,2) -- (0,2);
\draw[thick,rounded corners=15pt] (0,2) -- (2,2) -- (2,-2) -- (0,-2);
\put(2.3,0){\makebox(0,0){$\Gamma_2$}};
\put(1.75,0){\makebox(0,0){$\Omega_2$}};

\draw[thin,rounded corners=15pt] (0,-1.5) -- (-1.5,-1.5) -- (-1.5,1.5) -- (0,1.5);
\draw[thin,rounded corners=15pt] (0,1.5) -- (1.5,1.5) -- (1.5,-1.5) -- (0,-1.5);
\put(0,0){\makebox(0,0){$\Omega_1$}};
\put(-1.25,0){\makebox(0,0){$S$}};
\draw[->,>=stealth] (2,1)--(2.5,1);
\put(2.3,1.3){\makebox(0,0){$\boldsymbol{n}$}};
\draw[->,>=stealth] (-1.5,1)--(-1.1,1);
\put(-1.1,1.3){\makebox(0,0){$-\boldsymbol{\nu}$}};
\draw[->,>=stealth] (-1.5,1)--(-1.9,1);
\put(-1.75,1.3){\makebox(0,0){$\boldsymbol{\nu}$}};
\end{tikzpicture}
\caption{Subdomain $\Omega_1$ (cell) enclosed within subdomain $\Omega_2$ (membrane). The surface $S = \partial\Omega_1$ separates the domains, and $\Gamma_1 = \emptyset$ while $\Gamma_2 = \partial\Omega_2 \setminus S$.}
\label{fig:case2}
}
\end{figure}
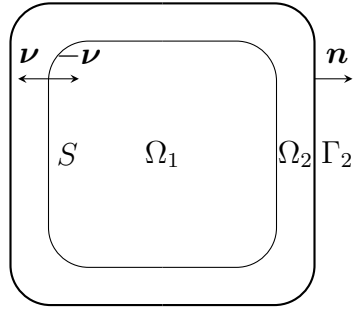
\end{itemize}
\vspace*{-0.2cm}
In this paper, we consider the following initial-boundary value problem of 
parabolic partial differential equations as a simplified model (cf.\ \cite{Sch00, Sch00b, Sch00c}): 
Find $u:=u(t,x)$, $\xi:=\xi(t,x)$, $v:=v(t,x)$, and $\psi:=\psi(t,x)$ satisfying  
\begin{align}
	\partial _t u - \Delta u + \xi+ \pi_1(u) =g_1, \quad \xi \in \beta (u) & 
	\quad {\rm in~} Q_1:=(0,T) \times \Omega_1, \label{heat1}\\
	\partial _t v - \kappa \Delta v + \psi + \pi_2 (v) =g_2, \quad \psi \in \beta (v) & 
	\quad {\rm in~} Q_2:=(0,T) \times \Omega_2, \label{heat2}\\
	\partial_{\boldsymbol{\nu}} u = \alpha(v-u) 
	& \quad {\rm on~} S_T:=(0,T) \times S, \label{jump1} \\
	\kappa \partial_{-\boldsymbol{\nu}} v = \alpha(u-v) 
	& \quad {\rm on~} S_T, \label{jump2} \\
	\partial_{\boldsymbol{n}} u = 0 
	& \quad {\rm on~} \Sigma_1:=(0,T) \times \Gamma_1, \label{bc1}\\
	\kappa \partial_{\boldsymbol{n}} v = 0 
	& \quad {\rm on~} \Sigma_2:=(0,T) \times \Gamma_2, \label{bc2}\\
	u(0) = u_{0\alpha} 
	& \quad {\rm in~} \Omega_1, \label{ini1} \\
	v(0) = v_{0\alpha} 
	& \quad {\rm in~} \Omega_2, \label{ini2}
\end{align}
where $\partial _t u:=\partial u/\partial t$ be the time derivative and 
$\Delta u:=\sum_{i=1}^d \partial ^2 u/\partial x_i^2$ stands for the {L}aplacian;  
$\kappa>0$ be a diffusion coefficient for the subdomain $\Omega_2$; 
$\beta: \mathbb{R} \to 2^{\mathbb{R}}$ be monotone graph possibly multivalued; 
$\pi_1, \pi_2 :\mathbb{R} \to \mathbb{R}$ be 
{L}ipschitz continuous functions, the most popular choice is $\beta(r):=r^3$ and 
$\pi_1(r)=\pi_2(r):=-r$ for $r \in\mathbb{R}$ which corresponds to the {A}llen--{C}ahn equations; 
$\boldsymbol{n}:=(n_1,n_2,\ldots,n_d):\Gamma_i \to \mathbb{R}^d$ be the normal vector outward from $\Omega$, 
$\boldsymbol{\nu}:S \to \mathbb{R}^d$ be the normal vector outward from $\Omega_1$ to $\Omega_2$, 
respectively. Using them, $\partial_{\boldsymbol{n}} u:=\nabla u\cdot \boldsymbol{n}$ 
is defined as the normal derivative, where $\nabla u:=(\partial u/ \partial x_1,\partial u/ \partial x_2,\ldots, \partial u/ \partial x_d)$.
Analogously, $\partial _{\boldsymbol{\nu}} u$ and 
$\partial _{-\boldsymbol{\nu}} v$ are defined;
$g_i: Q_i \to \mathbb{R}$, 
$u_{0\alpha}:\Omega_1 \to \mathbb{R}$, and $v_{0\alpha}: \Omega_2 \to \mathbb{R}$ are given functions. 
This kind of problem is well-known as the transmission problem. Indeed, 
the unknown $u$ in $Q_1$ satisfies the main equation \eqref{heat1} with 
the initial condition \eqref{ini1} and the {N}eumann boundary condition \eqref{bc1}. 
Concerning the transmission condition on $S$, 
\eqref{jump1} is a type of {R}obin boundary condition: 
\begin{equation*}
	\partial_{\boldsymbol{\nu}} u + \alpha u = \alpha v, \quad {\rm i.e.,~} \quad
	\partial_{\boldsymbol{\nu}} u = \alpha(v-u) \quad {\rm on~} S_T.
\end{equation*}
This condition can be interpreted from the viewpoint of domain $\Omega_1$ as follows. 
If the concentration $v$ in the exterior domain $\Omega_2$ exceeds the interior concentration $u$, then inflow takes place across the surface $S$. 
Conversely, outflow occurs when the exterior concentration is lower. 
The magnitude of this transport is governed by the nonnegative constant $\alpha$, which characterizes the permeability coefficient.
Condition \eqref{jump2} describes the same phenomenon from the reverse standpoint.
As a remark,  
merging \eqref{jump1} and \eqref{jump2} we also uncover
\begin{equation}
	\partial_{\boldsymbol{\nu}} u = - \kappa \partial_{-\boldsymbol{\nu}} v= \kappa \partial_{\boldsymbol{\nu}} v
	\quad {\rm on~} S_T. \label{neu}
\end{equation}
This is interpreted as a non-homogeneous {N}eumann boundary condition for $u$ on $S$. 
Therefore, it seems overdetermined. However, 
the function $v$ is also unknown in the right hand side of
\eqref{jump1} and \eqref{neu}. This $v$ is also unknown which is determined by the similar system 
\eqref{heat2}, \eqref{jump2}, \eqref{bc2}, and \eqref{ini2}. 
Therefore, this system is not overdetermined. 
\smallskip

The most important parameter is the permeability coefficient $\alpha$. 
Indeed, we consider the limiting procedures
$\alpha \to 0$ and $\alpha \to +\infty$ as interesting situations, more precisely, we will discuss 
in this paper
\begin{enumerate}
\item[(1)] well-posedness for $\alpha \in (0, +\infty)$; 
\item[(2)] asymptotic analysis as $\alpha \to 0$ and $\alpha \to +\infty$; 
\item[(3)] relationship with {M}osco convergence; 
\item[(4)] well-posedness for the case where $\alpha:=\alpha(t)$, i.e., when $\alpha$ depends on the time variable; 
\item[(5)] well-posedness for the case where the domain of the system undergoes drastic changes, i.e., topological changes as $\alpha(t) \to +\infty$ when $t \to T^*<T$. 
\end{enumerate}   
In the first context, the original system of transmission problems corresponds to \emph{gap junctions} with the permeability $\alpha \in (0,+\infty)$. 
In the second, when $\alpha \to 0$, the target $(u,v)$ of solutions $(u_\alpha, v_\alpha)$ obtained in the first context  satisfies two systems in $\Omega_1$ and $\Omega_2$ that are completely independent. 
This corresponds to that junctions are closed. 
When $\alpha \to +\infty$, the system of partial differential equations merges into a single 
equation in $\Omega=\Omega_1 \cup S \cup \Omega_2$ and the target is the 
strong solution of the heat equation on $\Omega$ if $\kappa=1$. 
In this sense, the asymptotic analysis 
$\alpha \to +\infty$ corresponds to the single structure. 
See also \cite{CD96, GL19} 
for the asymptotic analysis of the single problem with mixed boundary conditions of {N}eumann--{R}obin type to 
{D}irichlet--{N}eumann type.
In the last, we consider a dynamics where the domain initially consists of two regions $\Omega_1$
and $\Omega_2$, that are in contact with each other. As the dynamics evolves, 
these contact regions merge into a single domain, 
undergoing a topological change to become $\Omega$ in a sense. 
\smallskip

In the previous research by {H}\'edy {A}ttouch \cite{Att84}, as an asymptotic analysis of the above three types of problems, he considered the zero thickness limit $\varepsilon$ and the zero constant limit of the diffusion coefficient $\lambda$ (in this paper, we will use $\lambda$ as a different approximation parameter) within the thin film region. 
In a transmission problem involving two subdomains and the thin film region sandwiched between them, he investigated the respective asymptotic behaviors according to the limiting value $\alpha$ of the ratio $\lambda/\varepsilon$. The reason for choosing $\alpha$ as the variable for permeability in this paper comes from this work. 
Therefore, the fundamental and significant research has already been established by \cite{Att84}.
The present paper extends this analysis by investigating the interrelations among the three problems through the asymptotic behavior of the parameter $\alpha$. 
\smallskip

At the end of this introduction, we discuss the relationship between dynamic boundary conditions, which have been actively studied in recent years, and related areas. 
Dynamic boundary conditions refer to boundary conditions that include time derivatives. 
It is well established in the literature \cite{CR90, GLR23, Lie13, Miu25} that dynamic boundary conditions arise when taking the zero thickness limit in the domain $\Omega_2$ 
shown in 
{\bf Fig.}~\ref{fig:case2} 
for transmission problems, and such problems are closely related to this type of analysis. 
In this sense, this paper is based on the essential motivation of the study \cite{CFL19} on the {A}llen--{C}ahn equation with dynamic boundary conditions. 
Focusing on the {C}ahn--{H}illiard equation, in particular, three types of problems have attracted attention. 
The {GMS} model \cite{CF15, Gal06, GMS11}, the {LW} model \cite{CFW20, LW19}, and the {KLLM} model  \cite{KLLM21}, which is an intermediate problem between them. 
While we defer the derivation of each model to previous studies, one notable point is that in the {KLLM} model, a {R}obin type boundary condition is used to express the {KLLM} model positioned as an intermediate case. 
Similar to the parameter $\alpha$ representing permeability mentioned above, using a single parameter, it was proven in \cite{KLLM21} that asymptotic analysis with respect to $0$ and $+\infty$ leads to convergence to the {LW} model and {GMS} model, respectively. 
This shares precisely the same perspective as the asymptotic analysis that this paper focuses 
on---or rather, it draws inspiration from \cite{KLLM21} (see also the conclusion of this paper).

\section{Mathematical formulation of the problem}

In this section, 
we present the mathematical framework for the problem under consideration. 
Let us begin by establishing the notation and preliminaries that will be employed throughout this work. 

\subsection{Notation and Preliminaries}
We use the standard notation for function spaces. Let $\Omega \subset \mathbb{R}^d$ be a bounded domain and its boundary $\partial \Omega$ is at least $C^{0,1}$-class. 
In the following, sets are written as $\Omega$, but the notation is used by replacing it with $\Omega_i$ according to each context. 
Let $L^2(\Omega)$, $H^m(\Omega)$ for $m=1,2$ be the {L}ebesgue and {S}obolev spaces with the standard norm 
$\| \cdot \|_{X}$ and inner product $(\cdot, \cdot)_{X}$ with $X:=L^2(\Omega)$, $H^1(\Omega)$, and  $H^2(\Omega)$, respectively. 
Hereafter, we use the notation for norm and inner product corresponding to the {B}anach and 
{H}ilbert space $X$. Then, the dense and compact imbeddings $H^2(\Omega) \hookrightarrow H^1(\Omega) \hookrightarrow L^2(\Omega)$ hold. 
\smallskip

Next, we recall useful facts for the trace theory, which 
can be found in many literatures (see, e.g., \cite{BG87, GR86, HW21, Leo23, Nec67, Wlo87}). 
There exists 
a unique linear continuous operator (first trace) $\gamma_0: H^1(\Omega) \to H^{1/2}(\partial \Omega)$ such that 
\begin{equation*}
	\gamma_0 u = u_{|_{\partial \Omega}} \quad {\rm on~}S, 
	\quad {\rm for~all~} u \in C^{\infty}(\overline{\Omega}) \cap H^1(\Omega), 
\end{equation*}
where, $u_{|_{\partial \Omega}} $ stands for the restriction of $u$ to $\partial \Omega$, and 
we use the fractional {S}obolev space $H^{1/2}(\partial \Omega)$ for the trace theory. 
Moreover, there exists a linear continuous operator (it is called the extension or recovery of trace) ${\mathcal R}: H^{1/2}(\partial \Omega) \to H^1(\Omega)$ such that
\begin{equation*}
	\gamma_0 {\mathcal R} w = w 
	\quad {\rm for~all~} w \in H^{1/2} (\partial \Omega). 
\end{equation*}
In many cases, for $u \in H^1(\Omega)$ we simply denote $\gamma_0 u$ by $u$ not only 
in $\Omega$ but also on $\partial \Omega$. 
Next, fact concerns the normal component of trace from $\boldsymbol{L}^2_{\rm div}(\Omega):=\{ \boldsymbol{u} \in \boldsymbol{L}^2(\Omega):=L^2(\Omega)^d : {\rm div} \boldsymbol{u} \in L^2(\Omega)\}$ 
as the third trace theory, where ${\rm div} \boldsymbol{u}:=\sum_{i=1}^d \partial u_i/\partial x_i$. 
There exists 
a unique linear continuous operator (third trace) $\gamma_{\rm N}: \boldsymbol{L}^2_{\rm div}(\Omega) \to H^{-1/2}(\partial \Omega)$ such that 
\begin{equation*}
	\gamma_{\rm N} \boldsymbol{u} = (\boldsymbol{u}\cdot \boldsymbol{n})_{|_{\partial \Omega}} 
	\quad {\rm for~all~} \boldsymbol{u} \in \boldsymbol{C}^{\infty}(\overline{\Omega}) \cap \boldsymbol{L}^2_{\rm div}(\Omega), 
\end{equation*}
where $\boldsymbol{n}$ is the outward normal vector, (the case of this paper, 
$\boldsymbol{n}$ or $\boldsymbol{\nu}$ depends on the situation).  
We are also interested in $H^{-1/2}(\partial \Omega)$, the dual space of $H^{1/2}(\partial \Omega)$ equipped with the dual norm 
\begin{equation*}
	\| w^* \|_{H^{-1/2}(\partial \Omega)} := \sup_{\scriptsize \substack{ w \in H^{1/2}(\partial \Omega) \\ \|w\|_{H^{1/2}(\partial \Omega)}  = 1} } 
	\bigl| \langle w^* , w \rangle_{H^{-1/2}(\partial \Omega), H^{1/2}(\partial \Omega)}\bigr| \quad {\rm for~} w^* \in H^{-1/2}(\partial \Omega).  
\end{equation*}
In many cases, we simply denote $\gamma_{\rm N} \boldsymbol{u}$ by $\boldsymbol{u} \cdot \boldsymbol{n}$.  Thanks to this, for the normal derivative we see that 
if $u \in H^1(\Omega)$ and $ \Delta u (={\rm div} \nabla u) \in L^2(\Omega)$, then the following generalized {G}reen formula holds: 
\begin{gather}
	\langle \gamma_{\rm N} \nabla u , y \rangle_{H^{-1/2}(\partial \Omega), H^{1/2}(\partial \Omega)}= 
	\int_\Omega \Delta u 
	y\dx
	+
	\int_{\Omega} \! \nabla u \cdot \nabla y \dx  \quad {\rm for~} y \in H^1(\Omega), 
	\label{Green}\\
	\langle \gamma_{\rm N} \nabla u , w \rangle_{H^{-1/2}(\partial \Omega), H^{1/2}(\partial \Omega)}= 
	\int_\Omega \Delta u 
	{\mathcal R} w \dx
	+
	\int_{\Omega} \! \nabla u \cdot \nabla {\mathcal R} w \dx  \quad {\rm for~} w \in H^{1/2}(\partial \Omega). 
	\label{Greenw}
\end{gather}
While notation $\partial_{\boldsymbol{n}} u$ is generally preferred as 
the normal derivative, expression $\gamma_{\rm N} \nabla u$ is deliberately employed here in order to 
distinct the element $\partial_{\boldsymbol{n}} u \in L^2(\partial \Omega)$ or $\gamma_{\rm N} \nabla u \in H^{-1/2}(\partial \Omega)$, 
(see, e.g., \cite[Corollary 2.6]{GR86}, \cite[Lemma 5.1.1]{HW21}, or \cite[Th\'eor\`em 1.1, Chapitre 3]{Nec67}). 
The above discussion holds not only for $\Omega$ but also when replaced with $\Omega_1$ or $\Omega_2$. 
\smallskip

In order to discuss the mixed boundary condition for Case~1 
(see, {\sc Figure} \ref{fig:case1}), 
we additionally need to recall the trace space on an open surface $S \subset \partial \Omega_1$ (resp.\ $\Gamma_1\subset \partial \Omega_1$). 
Hereafter, we assume that $\partial S(=\partial \Gamma_i)$ is $C^{0,1}$-class.   
Follows from \cite[pp.189--190]{HW21}, we define {L}ions--{M}agenes--{S}trichartz space 
(see, also \cite{Gri85, Str67, LM70})
\begin{equation*}
	\tilde{H}^{1/2}(S)
	:=
	\bigl\{ w \in H^{1/2}(\partial \Omega_1) : {\rm supp} w \subset \overline{S} \bigr\}
	=
	\bigl\{ w \in L^2(S) : \tilde{w} \in H^{1/2}(\partial \Omega_1) \bigr\}
	 \bigl(= H^{1/2}_{00}(S)\bigr)
\end{equation*}
as a subspace of $H^{1/2}(\partial \Omega_1)$,  
where $\tilde{w}$ is a natural $0$-extension to $\partial \Omega_1$. The norm is defined by 
\begin{equation*}
	\| w \|_{\tilde{H}^{1/2}(S)} := \| \tilde{w} \|_{H^{1/2}(\partial \Omega_1)} 
	\quad {\rm for~} w \in \tilde{H}^{1/2}(S). 
\end{equation*}
Then, we see that $\tilde{H}^{1/2}(S) \subset H^{1/2}(\partial \Omega_1)$ holds. 
We also define the space 
\begin{gather*}
	H^{1/2} (S) := \bigl\{ w_{|_S} : w \in H^{1/2} (\partial \Omega_1) \bigr\}
\end{gather*}
equipped with the norm 
\begin{equation*}
	\| w \|_{H^{1/2}(S)} := \inf _{\scriptsize
	\substack{\bar{w} \in H^{1/2} (\partial \Omega_1)\\ \bar{w}_{|_S}=w }} \| \bar{w} \|_{H^{1/2}(\partial \Omega_1)} 
	\quad {\rm for~} w \in H^{1/2}(S). 
\end{equation*}
Then, we see that $\tilde{H}^{1/2}(S) \subset H^{1/2}(S)$ holds. 
Note that while the notation $(H^{1/2}_{00}(S))'$ is sometimes used to represent the dual space of $\tilde{H}^{1/2}(S)$, in this paper following \cite[pp.189--190]{HW21} we denote $H^{-1/2}(S)$ as the dual space of $\tilde{H}^{1/2}(S)(=H_{00}^{1/2}(S))$. 
Finally we see that 
if $w^* \in H^{-1/2}(\partial \Omega_1)$, then $w^*$ makes sense as the element of $H^{-1/2}(S)$ 
as follows:
\begin{equation}
	\bigl\langle (w^*)_{|_S}, w \bigr\rangle_{H^{-1/2}(S), 
	\tilde{H}^{1/2}(S)}:=\langle w^*, \tilde{w} 
	\rangle_{H^{-1/2}(\partial \Omega_1), H^{1/2}(\partial \Omega_1)} 
	\quad {\rm for~} w \in \tilde{H}^{1/2}(S). 
	\label{howto}
\end{equation}
The symbol $(w^*)_{|_S}$ stands for the restriction of $w^*$ to $\tilde{H}^{1/2}(S)$. 
Of course the above definition is independent of $i=1,2$. 
More detail, we can find literatures 
\cite[pp.1.56--1.60]{BG87}, 
\cite[pp.961--964]{FG97}, 
\cite[p.33]{Gri85},  
\cite[pp.189--190]{HW21}, 
\cite[p.249]{Leo23}.  
As a remark, if we consider Case~2 (see, {\sc Figure}~\ref{fig:case2}), 
we do not need such an 
intricate settings, because if $S = \partial \Omega_1$, then $\tilde{H}^{1/2}(S)=H^{1/2}(S)$. 
\smallskip

\subsection{Transmission problem corresponding to gap junctions}

In this subsection, we will discuss the well-posedness for the problem corresponding to $\alpha \in (0,+\infty)$. The result is based on the standard theory of evolution equation governed by the 
subdifferential (see, e.g., \cite{Att84, Bar10, Bre73}). 
\smallskip

Throughout of this paper, we assume that:
\begin{enumerate}
\item[(A0)] $\Omega$, $\Omega_1$, and $\Omega_2$ are bounded and at least $C^{0,1}$-class, $\partial S$ is also $C^{0,1}$-class; 
\item[(A1)] $\beta$ is a maximal monotone graph in $\mathbb{R} \times \mathbb{R}$, which
coincides with the subdifferential $\beta = \partial j$ of 
some proper, lower semicontinuous, and convex function 
$j: \mathbb{R} \to [0,+\infty]$ such that 
$j(0)=0$, with the effective domain $D(\beta) \subset \mathbb{R}$; 
\item[(A2)] $\pi_i:\mathbb{R} \to \mathbb{R}$ are {L}ipschitz continuous with {L}ipschitz constants $L_i>0$, and 
$\pi_i(0)=0$ hold for $i=1,2$; 
\item[(A3)] $g_i \in L^2(0,T;L^2(\Omega_i))$ for $i=1,2$, $u_{0\alpha} \in H^1(\Omega_1)$ with 
$j(u_{0\alpha}) \in L^1(\Omega_1)$ 
 and $v_{0\alpha} \in H^1(\Omega_2)$ with $j(v_{0\alpha}) \in L^1(\Omega_2)$. 
\end{enumerate}

The essential assumption {\rm (A0)} means that the triple junction is not a cusp in 
{\sc Figure}~\ref{fig:case1}. 
\smallskip

First, regarding the interactions within cells connected by \emph{gap junction}, the transmission problem for the following system of partial differential equations can be solved: 

\begin{proposition}\label{pro} Let $\alpha \in (0,+\infty)$ be fixed and assume {\rm (A0)}--{\rm (A3)} hold. 
Then, there exists a unique quadruplet $(u_\alpha, \xi_\alpha, v_\alpha, \psi_\alpha)$ of functions 
\begin{gather*}
	u_\alpha \in H^1\bigl(0,T; L^2(\Omega_1) \bigr) \cap L^\infty\bigl(0,T; H^1(\Omega_1) \bigr),
	\quad 
	\Delta u_\alpha, \xi_\alpha \in L^2\bigl(0,T; L^2(\Omega_1) \bigr), \\
	v_\alpha \in H^1\bigl(0,T; L^2(\Omega_2) \bigr) \cap L^\infty\bigl(0,T; H^1(\Omega_2) \bigr),
	\quad 
	\Delta v_\alpha, \psi_\alpha \in L^2\bigl(0,T; L^2(\Omega_2) \bigr)
\end{gather*}
such that 
\begin{align}
	\partial _t u_\alpha - \Delta u_\alpha + \xi_\alpha + \pi_1(u_\alpha) =g_1, \quad 
	\xi_\alpha \in \beta(u_\alpha) & 
	\quad {\it a.e.~in~} Q_1, \label{heat1a}\\
	\partial _t v_\alpha - \kappa \Delta v_\alpha + \psi_\alpha+ \pi_2 (v_\alpha) =g_2, \quad 
	\psi_\alpha \in \beta(v_\alpha) & 
	\quad {\it a.e.~in~} Q_2, \label{heat2a}\\
	\partial_{\boldsymbol{\nu}} u_\alpha = \alpha(v_\alpha-u_\alpha) 
	& \quad {\it a.e.~on~} S_T, \label{jump1a}\\
	\kappa \partial_{-\boldsymbol{\nu}} v_\alpha = \alpha(u_\alpha-v_\alpha) 
	& \quad {\it a.e.~on~} S_T, \label{jump2a}\\
	\partial_{\boldsymbol{n}} u_\alpha = 0 
	& \quad {\it a.e.~on~} \Sigma_1, \label{bc1a}\\
	\kappa \partial_{\boldsymbol{n}} v_\alpha = 0 
	& \quad {\it a.e.~on~} \Sigma_2, \label{bc2a}\\
	u_\alpha(0) = u_{0\alpha} 
	& \quad {\it a.e.~in~} \Omega_1, \label{ini1a}\\
	v_\alpha(0) = v_{0\alpha} 
	& \quad {\it a.e.~in~} \Omega_2. \label{ini2a}
\end{align}
\end{proposition}
\smallskip

The proof is quite standard. 
The idea is based on the abstract theory of evolution equation on a {H}ilbert space 
$\mathcal H:=L^2 (\Omega_1)\times L^2 (\Omega_2)$. 
Define the inner product
\begin{equation*}
	(U,Y)_{\mathcal H}:=(u,y)_{L^2(\Omega_1)} + (v,z)_{L^2(\Omega_2)}
\end{equation*}
for $U:=(u,v)$ and $Y:=(y,z)$. 
Of course, we can identify 
${\mathcal H}$ by $L^2(\Omega)$ interpreting $U$ by $\tilde{u}+\tilde{v}$. As a remark, 
this identification doesn't work for the case where 
$H^1(\Omega_1) \times H^1(\Omega_2)$, because $\tilde{u}+ \tilde{v} \notin H^1(\Omega)$ even if 
$u \in H^1(\Omega_1)$ and $v \in H^1(\Omega_2)$. Now, for $\alpha>0$ 
we define a proper, lower semi continuous, and convex functional 
$\varphi_\alpha: {\mathcal H} \to [0,+\infty]$ by 
\begin{equation}
	\varphi_\alpha(U)\! := \! 
	\begin{cases}
	\displaystyle 
	\frac{1}{2} \int_{\Omega_1} \! |\nabla u|^2 \dx + 
	\frac{\kappa}{2} \int_{\Omega_2} \! |\nabla v|^2 \dx + 
	\frac{\alpha}{2} \int_{S} |u-v|^2 \dS 
	& {\rm if~} 
	U
	\in H^1(\Omega_1) \! \times \! H^1(\Omega_2)
	, \\
	+\infty & {\rm otherwise,}
	\end{cases}\label{phi_a}
\end{equation}
that is, $D(\varphi_\alpha)=H^1(\Omega_1) \times H^1(\Omega_2)$ which is independent of $\alpha>0$ 
(see, also \cite{GL19}). 
Then, we have the following characterization of the subdifferential, this is a key of the proof 
for {\rm Proposition} \ref{pro}: 

\begin{lemma} \label{chara}
Let $U:=(u,v) \in D(\varphi_\alpha)$ and $U^*:=(u^*,v^*) \in {\mathcal H}$. Then, 
$U^* \in \partial \varphi_\alpha(U)$ in ${\mathcal H}$ if and only if 
\begin{align}
	u^* =- \Delta u & 
	\quad {\it a.e.~in~} \Omega_1, \label{equ}\\
	v^* = - \kappa \Delta v & 
	\quad {\it a.e.~in~} \Omega_2, \label{eqv}\\
	\partial_{\boldsymbol{\nu}} u = \alpha(v-u) 
	& \quad {\it a.e.~on~} S, \label{traSu}\\
	\kappa \partial_{-\boldsymbol{\nu}} v = \alpha(u-v) 
	& \quad {\it a.e.~on~} S, \label{traSv}\\
	\partial_{\boldsymbol{n}} u = 0 
	& \quad {\it a.e.~on~} \Gamma_1, \label{bcu}\\
	\kappa \partial_{\boldsymbol{n}} v = 0 
	& \quad {\it a.e.~on~} \Gamma_2 \label{bcv}.
\end{align}
This means that
\begin{equation*}
	D(\partial \varphi_\alpha)=\left\{ 
	(y, z)
	\in H^1(\Omega_1) \times H^1(\Omega_2) : 
	\begin{gathered}
	\Delta y \in L^2(\Omega_1), \ \Delta z \in L^2 (\Omega_2), \\
	\partial_{\boldsymbol{\nu}} y = \alpha(z-y), \
	\kappa \partial_{-\boldsymbol{\nu}} z = \alpha(y-z) 
	\ {\it a.e.~on~} S, \\
	\partial_{\boldsymbol{n}} y = 0 \ {\it a.e.~on~} \Gamma_1,
	\
	\partial_{\boldsymbol{n}} z = 0 \ {\it a.e.~on~} \Gamma_2.
	\end{gathered}
	\right\}.
\end{equation*}
\end{lemma}

\begin{proof} 
Let $U \in D(\varphi_\alpha)$ and assume 
$U^*:=(u^*,v^*)  \in \partial \varphi_\alpha(U)$ in ${\mathcal H}$. This implies 
$U \in D(\partial \varphi _\alpha):=\{ U \in D(\varphi_\alpha) : \partial \varphi_\alpha (U) \ne \emptyset \}$. 
From the definition of the subdifferential, we have 
\begin{equation*}
(U^*, Y-U)_{\mathcal H} \le \varphi_\alpha(Y)-\varphi_\alpha(U) \quad {\rm for~all~} Y:=(y,z)  \in {\mathcal H}. 
\end{equation*}
Let $Y \in D(\varphi_\alpha)$ and $\varepsilon>0$, taking $U + \varepsilon Y \in D(\varphi_\alpha)$ as the test function in the above, we get 
\begin{align*}
	& 
	(u^*, \varepsilon y )_{L^2(\Omega_1)} + (v^*, \varepsilon z )_{L^2(\Omega_2)} = 
	(U^*, \varepsilon Y)_{\mathcal H} \\
	& {} \le \frac{1}{2} \int_{\Omega_1} 
	\bigl| \nabla (u+\varepsilon y) \bigr|^2 \dx -
	\frac{1}{2} \int_{\Omega_1}
	|\nabla u|^2 \dx
	+ \frac{\kappa}{2} \int_{\Omega_2}
	\bigl| \nabla (v+\varepsilon z) \bigr|^2 \dx - 
	\frac{\kappa}{2} \int_{\Omega_2}
	|\nabla v|^2 \dx \\
	& \quad {}
	+
	\frac{\alpha}{2} \int_{S} 
	\bigl| (u+\varepsilon y )-(v+\varepsilon z) \bigr|^2 
	\dS
	-
	\frac{\alpha}{2} \int_{S} 
	|u-v|^2\dS
	\\
	& {} = \varepsilon \int _{\Omega_1} \!  \nabla u \cdot \nabla y \dx 	
	+ \frac{\varepsilon^2}{2} \int_{\Omega_1}|\nabla y|^2 \dx 
	+ \varepsilon \kappa \int _{\Omega_2} \!  \nabla v \cdot \nabla z \dx 	
	+ \frac{\varepsilon^2\kappa}{2} \int_{\Omega_2}|\nabla z|^2 \dx \\
	& \quad {} + \varepsilon \alpha \int _{S} (u-v)(y-z) \dS 	
	+ \frac{\varepsilon^2 \alpha}{2} \int_{S} |y-z|^2 \dS. 
\end{align*}
Here, dividing the above by $\varepsilon$ and letting $\varepsilon \to 0$, we obtain the 
inequality 
\begin{equation*}
	(u^*,  y )_{L^2(\Omega_1)} + (v^*, z )_{L^2(\Omega_2)}
	\le \int _{\Omega_1} \! \! \nabla u \cdot \nabla y \dx 	
	+ \kappa \int _{\Omega_2} \! \! \nabla v \cdot \nabla z \dx 	
	+ \alpha \int _{S} (u-v)(y-z) \dS. 
\end{equation*}
Analogously, taking $U-\varepsilon Y \in D(\varphi_\alpha)$ as 
the test function in the above calculation, we obtain the opposite inequality in the above.   
Thus, we finally deduce 
\begin{equation}
	(u^*,  y )_{L^2(\Omega_1)} + (v^*, z )_{L^2(\Omega_2)}
	= \int _{\Omega_1} \! \! \nabla u \cdot \nabla y \dx 	
	+ \kappa \int _{\Omega_2} \! \! \nabla v \cdot \nabla z \dx 	
	+ \alpha \int _{S} (u-v)(y-z) \dS
	\label{weak}
\end{equation}
for all $Y \in D(\varphi_\alpha)$. 
From now, we prove \eqref{equ}--\eqref{bcv} step by step. Firstly, 
for all $y \in {\mathcal D}(\Omega_1)(= C_{0}^\infty(\Omega_1))$ and $z \equiv 0$ in \eqref{weak}, 
we see that $u^* = -\Delta u$ in ${\mathcal D}'(\Omega_1)$ (in the sense of distribution). 
Moreover, 
we know $u ^* \in L^2(\Omega_1)$, therefore from the comparison in the equation we see 
that $u^* = -\Delta u$ in $L^2(\Omega_1)$ and \eqref{equ} holds. 
Analogously, we see 
that $v^* = -\kappa \Delta v$ in $L^2(\Omega_2)$ and \eqref{eqv} holds. 
Secondly, 
for all $y \in H^1(\Omega_1)$ and $z \equiv 0$ in \eqref{weak}, and using \eqref{Green} we have 
\begin{equation*}
	\langle \gamma_{\rm N} \nabla u, y \rangle_{H^{-1/2}(\partial \Omega_1), H^{1/2} (\partial \Omega_1)}
	= - \alpha \int _{S} (u-v) y \dS 
	= \alpha \int _{\partial \Omega_1} (\tilde{v}-\tilde{u}) y \dg,
\end{equation*}
as a remark, $\tilde{v}$ means natural $0$-extension of the trace $\gamma v$ on $S$ to $\partial \Omega_1$. 
Therefore, by the comparison in the equation, we see that 
\begin{equation}
	\bigl( \partial_{\boldsymbol{n}} u= \bigr) \ \gamma_{\rm N} \nabla u =\alpha (\tilde{v}-\tilde{u}) \quad {\rm in~} L^2(\partial \Omega_1). 
	\label{NDrep} 
\end{equation}
Thirdly, recalling \eqref{howto} we define $(\gamma_{\rm N} \nabla u)_{|_{\Gamma_1}}\in H^{-1/2}(\Gamma_1)$ and 
$(\gamma_{\rm N} \nabla u)_{|_{S}} \in H^{-1/2}(S)$ as follows: 
\begin{gather*}
	\bigl \langle (\gamma_{\rm N} \nabla u)_{|_{\Gamma_1}}, w \bigr\rangle_{H^{-1/2}(\Gamma_1), 
	\tilde{H}^{1/2}(\Gamma_1)}:=\langle \gamma_{\rm N} \nabla u, \tilde{w} 
	\rangle_{H^{-1/2}(\partial \Omega_1), H^{1/2}(\partial \Omega_1)} 
	\quad {\rm for~} w \in \tilde{H}^{1/2}(\Gamma_1), \\
	\bigl\langle (\gamma_{\rm N} \nabla u)_{|_S}, w \bigr\rangle_{H^{-1/2}(S), 
	\tilde{H}^{1/2}(S)}:=\langle \gamma_{\rm N} \nabla u, \tilde{w} 
	\rangle_{H^{-1/2}(\partial \Omega_1), H^{1/2}(\partial \Omega_1)} 
	\quad {\rm for~} w \in \tilde{H}^{1/2}(S).
\end{gather*}
Then, for all $w \in \tilde{H}^{1/2}(\Gamma_1)$ and $z \equiv 0$ in \eqref{weak}, 
we know $\tilde{w} \in H^{1/2}(\partial \Omega_1)$ from the definition. Therefore, ${\mathcal R} \tilde{w} \in H^1(\Omega_1)$, $\gamma {\mathcal R} \tilde{w}=0$ a.e.\ in $S$, and 
\begin{equation*}
	- \int_{\Omega_1} \Delta u  {\mathcal R} \tilde{w} \dx  
	= \int _{\Omega_1} \! \nabla u \cdot \nabla {\mathcal R} \tilde{w} \dx 		
	+ \alpha \int _{S} (u-v) \gamma {\mathcal R} \tilde{w} \dS =0.
\end{equation*}
By using \eqref{Greenw}
\begin{equation*}
	\bigl \langle (\gamma_{\rm N} \nabla u)_{|_{\Gamma_1}}, w \bigr\rangle_{H^{-1/2}(\Gamma_1), 
	\tilde{H}^{1/2}(\Gamma_1)}:=\langle \gamma_{\rm N} \nabla u, \tilde{w} 
	\rangle_{H^{-1/2}(\partial \Omega_1), H^{1/2}(\partial \Omega_1)} =0,
\end{equation*}
that is, $(\gamma_{\rm N} \nabla u)_{|_{\Gamma_1}}=0$ in $H^{-1/2}(\Gamma_1)$ namely in $L^2(\Gamma_1)$ by the comparison in the equation. Analogously, we obtain 
$(\gamma_{\rm N} \nabla u)_{|_S}=\alpha(v-u)$ in $L^2(S)$. 
Hereafter, we can write $(\gamma_{\rm N} \nabla u)_{|_{\Gamma_1}}=\partial _{\boldsymbol{n}} u$, 
$(\gamma_{\rm N} \nabla u)_{S} = \partial _{\boldsymbol{\nu}} u$, and compare with the representation 
\eqref{NDrep}, we finally conclude
\begin{equation*}
	\partial _{\boldsymbol{n}} u = 0 \quad {\rm in~}L^2(\Gamma_1), \quad 
	\partial_{\boldsymbol{\nu}} u = \alpha (v-u) \quad {\rm in~}L^2(S), 
\end{equation*}
that is, \eqref{traSu} and \eqref{bcu}. Analogously, \eqref{traSv} and \eqref{bcv} hold. 
\smallskip

Conversely, 
if $U \in H^1(\Omega_1)\times H^1(\Omega_2)$ satisfying \eqref{equ}--\eqref{bcv}. Then, using 
the integration by part  
we can easily prove that $U^*:=(u^*,v^*) =(-\Delta u, -\kappa \Delta v) $ satisfies 
\begin{equation*}
(U^*, Y-U)_{\mathcal H} \le \varphi_\alpha(Y)-\varphi_\alpha(U) \quad {\rm for~all~} Y:=(y,z)  \in D(\varphi_\alpha). 
\end{equation*}
Trivially, this also holds for $Y \in {\mathcal H} \setminus D(\varphi_\alpha)$. Thus, $U \in D(\partial \varphi_\alpha)$ and $U^* \in \partial \varphi_\alpha (U)$ can be obtained.   
\end{proof}
\smallskip

As a remark, in the case where the domain is enough smooth, (for example, $\Omega_1$ is $C^2$-class), 
the following two kind of elliptic estimates hold (see, e.g., \cite{BG87, Gri85, HW21, Nec67, Wlo87}): 
\begin{gather*}
	\| u \|_{H^2(\Omega_1)} \le C_{\rm e}\bigl( \| \Delta u \|_{L^2(\Omega_1)} 
	+ \| u  \|_{H^{3/2}(\partial \Omega_1)} + \| u\|_{L^2(\Omega_1)} \bigr), \\
	\| u \|_{H^2(\Omega_1)} \le C_{\rm e}\bigl( \| \Delta u \|_{L^2(\Omega_1)} 
	+ \| \partial_{\boldsymbol{n}} u  \|_{H^{1/2}(\partial \Omega_1)} + \| u\|_{L^2(\Omega_1)} \bigr).
\end{gather*}
It seems at first glance that, we can gain the $H^2$-regularities in Lemma~\ref{chara}. 
However, in \eqref{traSu} and \eqref{bcu}, we have neither 
$\alpha \tilde{v} - \tilde{\partial_{\boldsymbol{n}} u} \in H^{3/2}(\partial \Omega_1)$
nor  
$\alpha(\tilde{v}-\tilde{u}) \in H^{1/2}(\partial \Omega_1)$. 
This is not due to the smoothness of the domain, 
but rather due to the difficulty of the mixed boundary conditions.  
\smallskip

We recall the {Y}osida approximation
$\beta_\lambda$ of 
$\beta$ for the parameter $\lambda>0$ defined by 
\begin{gather*}
	\beta_\lambda (r):=\frac{1}{\lambda} \bigl( r-J_\lambda (r) \bigr) :=\frac{1}{\lambda} 
	\bigl( r-(I+\lambda \beta) ^{-1}(r)\bigr) \quad {\rm for~} r \in \mathbb{R}.
\end{gather*}
From the theory of maximal monotone (see, e.g., \cite{Bar10, Bre73}), 
we see that $\beta_\lambda$ is monotone and Lipschitz continuous. Moreover, 
from ({\rm A1}) we have $\beta_\lambda(0)=0$. 
Let $\lambda>0$. Apply the {Y}osida approximation $\beta_\lambda$ for $\beta$ in \eqref{heat1a} and  \eqref{heat2a}, respectively. 
From the characterization Lemma \ref{chara}, applying abstract theory of evolution equation 
\cite[Theorem~4.11]{Bar10} or \cite[Theoremes~3.4 and 3.6]{Bre73}
we can solve the {C}auchy problem: Find $U_\lambda:=(u_\lambda, v_\lambda)$ satisfying
\begin{equation}
	\begin{cases}
	U_\lambda'(t) + \partial \varphi_\alpha \bigl( U_\lambda(t) \bigr) 
	+ B_\lambda \bigl( U_\lambda(t) \bigr) = \bigl(g_1(t),g_2(t) \bigr)  \quad 
	{\rm in~}{\mathcal H}, \ {\rm for~a.a.~} t \in (0,T), \\
	U_\lambda(0)=(u_{0\alpha}, v_{0\alpha})  \quad {\rm in~}{\mathcal H}
	\end{cases} \label{ee}
\end{equation}
in the class $U_\lambda \in H^1(0,T; {\mathcal H}) \cap L^\infty(0,T;D(\varphi_\alpha)) \cap L^2(0,T;D(\partial \varphi_\alpha))$, 
where 
$B_\lambda(U_\lambda):=(\beta_\lambda(u_\lambda)+\pi_1(u_\lambda), \beta_\lambda(v_\lambda)+\pi_2(v_\lambda))$.  
We can apply the standard idea (see, e.g., \cite[pp.10--11, pp.105--108]{Bre73}, 
\cite[Proposition~3.1]{CF15a}, or \cite[Proposition~4.1]{CF15}) to treat the Lipschitz perturbation $B_\lambda$. The uniqueness is also 
standard, we omit its proof. 
Hereafter, for the each component $u_\lambda$ and $v_\lambda$
of the above solution $U_\lambda$, we 
obtain the uniform estimate independent of $\lambda>0$. 
\smallskip

\begin{lemma}\label{1st} There exist constants $M_1, M_2, M_3>0$ independent of $\lambda>0$ and 
$\alpha \in (0,+\infty)$ such that 
\begin{align}
	& \| u_\lambda \|_{L^\infty(0,T;L^2(\Omega_1))} 
	+ \| v_\lambda \|_{L^\infty(0,T;L^2(\Omega_2))}
	\le 
	\|u_{0\alpha}\|_{L^2(\Omega_1)} + 
	\|v_{0\alpha}\|_{L^2(\Omega_2)} + 
	M_1,\label{est1}\\
	& \| u_\lambda \|_{L^2(0,T;H^1(\Omega_1))} 
	+ \kappa \| v_\lambda \|_{L^2(0,T;H^1(\Omega_2))} 
	+ \sqrt{\alpha} \| u_\lambda-v_\lambda \|_{L^2(0,T;L^2(S))} 
	\notag \\
	& {} 
	\le 
	M_2 \bigl( 1+ 
	\|u_{0\alpha}\|_{L^2(\Omega_1)} + 
	\|v_{0\alpha}\|_{L^2(\Omega_2)} \bigr),\label{est1a}
	\\
	& \bigl\| \beta_\lambda (u_\lambda) \bigl\|_{L^2(0,T;L^2(\Omega_1))} 
	+
	\bigl\| \beta_\lambda (v_\lambda) \bigl\|_{L^2(0,T;L^2(\Omega_2))} 
	\notag \\
	& {} \le M_3\Bigl( 1+ \|u_{0\alpha}\|_{L^2(\Omega_1)} 
	+ 
	\|v_{0\alpha}\|_{L^2(\Omega_2)}
	+
	\bigl\| j(u_{0\alpha}) \bigr\| _{L^1(\Omega_1)}^{1/2}
	+
	\bigl\| j(v_{0\alpha}) \bigr\| _{L^1(\Omega_2)}^{1/2} \Bigr).
	\label{est1b}
\end{align}
\end{lemma}

\begin{proof}
Recall the equation of $u_\lambda$ corresponding to \eqref{heat1a}, \eqref{jump1a}, \eqref{bc1a},  and \eqref{ini1a}.
\begin{align}
	\partial _t u_\lambda - \Delta u_\lambda + \beta_\lambda(u_\lambda) + \pi_1(u_\lambda) =g_1 
	& \quad {\rm a.e.~in~} Q_1, \label{la1}\\
	\partial_{\boldsymbol{\nu}} u_\lambda = \alpha(v_\lambda-u_\lambda) 
	& \quad {\rm a.e.~on~} S_T, \label{la2}\\
	\partial_{\boldsymbol{n}} u_\lambda = 0 
	& \quad {\rm a.e.~on~} \Sigma_1, \label{la3}\\
	u_\lambda(0) = u_{0\alpha} 
	& \quad {\rm a.e.~in~} \Omega_1 \notag
\end{align}
from \eqref{ee}. 
The first estimate is obtained to take the solution as the test function. 
Multiplying \eqref{la1} by $u_\lambda$, integrating it over $\Omega_1$, and using 
\eqref{la2}--\eqref{la3} we get  
\begin{align*}
& \frac{1}{2} \frac{\d}{\dt} \| u_\lambda \|_{L^2(\Omega_1)}^2 
	+ 
	\int_{\Omega_1} | \nabla u_\lambda |^2 \dx  
	- \int_{S} \alpha(v_\lambda -u_\lambda) u_\lambda \dS 
	+ \bigl( \beta_\lambda ( u_\lambda), u_\lambda \bigr)_{L^2(\Omega_1)} 
	\notag \\
	 & {} = \bigl( g_1 - \pi_1(u_\lambda), u_\lambda \bigr)_{L^2(\Omega_1)}
\end{align*}
a.e.\ on $(0,T)$. 
Analogously, we obtain similar equality for $v_\lambda$, 
Then, summing up the resultant and using the monotonicity of $\beta_\lambda$ we obtain 
\begin{align*}
	& \frac{1}{2} \frac{\d}{\dt} \bigl\| u_\lambda(t) \bigr\|_{L^2(\Omega_1)}^2 
	+ 
	\frac{1}{2} \frac{\d}{\dt} \bigl\| v_\lambda(t) \bigr\|_{L^2(\Omega_2)}^2
	+ 
	\int_{\Omega_1} \bigl| \nabla u_\lambda(t) \bigr|^2 \dx  
	+
	\kappa \int_{\Omega_2} \bigl| \nabla v_\lambda(t) \bigr|^2 \dx  
	\\
	& \quad {} + \alpha \bigl\| u_\lambda(t)- v_\lambda(t) \bigr\|_{L^2(S)}^2 \\
	& {} \le \bigl( g_1(t) - \pi_1 \bigl(u_\lambda(t) \bigr), u_\lambda(t) \bigr)_{L^2(\Omega_1)}
	+  \bigl( g_2(t) - \pi_2 \bigl(v_\lambda(t) \bigr), v_\lambda(t) \bigr)_{L^2(\Omega_2)}
	\\
	& {} \le \frac{1}{2} \bigl\| g_1(t) \bigr\|_{L^2(\Omega_1)}^2 
	 + \left( \frac{1}{2} +L_1 \right) \bigl\| u_\lambda (t) \bigr\|_{L^2(\Omega_1)}^2
	 +\frac{1}{2} \bigl\| g_2(t) \bigr\|_{L^2(\Omega_2)}^2 
	 + \left( \frac{1}{2} +L_2 \right) \bigl\| v_\lambda (t) \bigr\|_{L^2(\Omega_2)}^2
\end{align*}
for a.a.\ $t \in (0,T)$, where we used the assumption {\rm (A2)}. 
Thus, applying the {G}ronwall inequality we deduce \eqref{est1} under the assumption {\rm (A3)}. 
Integrating the above $[0,T]$ with respect to time variable, and using 
\eqref{est1} we can also obtain the estimate \eqref{est1a}. 
\smallskip

Next, testing \eqref{la1} by $\beta_\lambda( u_\lambda)$, and using
\eqref{la2}--\eqref{la3}, we get 
\begin{align*}
	&
	\frac{\d}{\dt}\int_{\Omega_1} j_\lambda(u_\lambda) \dx
	+ \int_{\Omega_1} \beta_\lambda'(u_\lambda) |\nabla u_\lambda|^2 \dx 
	- \int_S \alpha (v_\lambda- u_\lambda) \beta_\lambda(u_\lambda) \dS 
	+ \bigl\| \beta_\lambda (u_\lambda) \bigr\|_{L^2(\Omega_1)}^2 \\
	& {} = \bigl( g_1 - \pi_1 (u_\lambda), \beta_\lambda(u_\lambda) \bigr)_{L^2(\Omega_1)}
\end{align*}
a.e.\ on $(0,T)$, where $j_\lambda$ is the {M}oreau--Yosida regularization of $j$ defined by 
\begin{equation}
	j_\lambda(r):=\inf_{s \in \mathbb{R}} \left\{ \frac{1}{2\lambda} |r-s|^2+j(s) \right\}
	= \frac{1}{2\lambda} \bigl| r-J_\lambda (r) \bigr|^2 + j\bigl( J_\lambda(r) \bigr)
\end{equation} 
for $r \in \mathbb{R}$. Then, we have $\beta_\lambda=j_\lambda'$ similar to $\beta =\partial j$. 
From the same way to the equation of $v_\lambda$ by $\beta_\lambda(v_\lambda)$, sum up and  
integrate $[0,\tau]$ to the resultant. Then, using the monotonicity of $\beta_\lambda$, 
in other word $\beta_\lambda' (r) \ge 0$ for all $r \in \mathbb{R}$, we deduce 
\begin{align}
	&
	\int_{\Omega_1} j_\lambda \bigl( u_\lambda(\tau) \bigr) \dx 
	 +
	\int_{\Omega_2} j_\lambda \bigl( v_\lambda(\tau) \bigr) \dx 
	+ \int_0^\tau \bigl\| \beta_\lambda (u_\lambda) \bigr\|_{L^2(\Omega_1)}^2 \dt
	+ \int_0^\tau \bigl\| \beta_\lambda (v_\lambda) \bigr\|_{L^2(\Omega_2)}^2 \dt
	\notag
	\\
	& {} \le 
	\bigl\| j(u_{0\alpha}) \bigr\| _{L^1(\Omega_1)}
	+
	\bigl\| j(v_{0\alpha}) \bigr\| _{L^1(\Omega_2)}
	+
	\|g_1 \|_{L^2(0,T;L^2(\Omega_1))}^2
	+ \|g_2 \|_{L^2(0,T;L^2(\Omega_2))}^2
	\notag \\
	& \quad {}
	+(L_1^2+L_2^2)M_1^2T
	+
	\frac{1}{2}
	 \int_0^\tau \bigl\| \beta_\lambda (u_\lambda) \bigr\|_{L^2(\Omega_1)}^2 \dt
	 +
	 \frac{1}{2}
	 \int_0^\tau \bigl\| \beta_\lambda (v_\lambda) \bigr\|_{L^2(\Omega_2)}^2 \dt
	 \label{btlam}
\end{align}
for all $\tau \in [0,T]$. 
Now, we used the 
property $0 \le j_\lambda(r) \le j(r)$ for $r \in \mathbb{R}$ 
of the {M}oreau--{Y}osida regularization $j_\lambda$ 
corresponding to the convex primitive $j$ (see, e.g., \cite{Bar10, Bre73, CF15a, CF15}). 
Thus, we conclude \eqref{est1b}. 
\end{proof}
\bigskip

In order to apply the limiting procedure not only $\lambda \to 0$ but also $\alpha \to 0$ or 
$\alpha \to + \infty$ we make sure the dependence of the bounds with respect to 
$\lambda>0$ and $\alpha \in (0,+\infty)$. 
\smallskip

\begin{lemma}\label{2nd} There exist constants $M_4, M_5>0$ independent of $\lambda>0$ and 
$\alpha \in (0,+\infty)$ such that 
\begin{align}
	& \| u_\lambda \|_{H^1(0,T;L^2(\Omega_1))} 
	+ \| v_\lambda \|_{H^1(0,T;L^2(\Omega_1))} 
	+ \| u_\lambda \|_{L^\infty(0,T;H^1(\Omega_2))}
	 \notag \\
	& \quad {}
	+ \kappa \| v_\lambda \|_{L^\infty(0,T;H^1(\Omega_2))}
	 + \sqrt{\alpha} \| u_\lambda-v_\lambda \|_{L^\infty(0,T;L^2(S))} 
	\notag \\
	& {} \le M_4 \Bigl( 1+ 
	\sqrt{\alpha} \| u_{0\alpha}-v_{0\alpha} \|_{L^2(S)} 
	+
	\| u_{0\alpha} \|_{H^1(\Omega_1)}
	+
	\kappa \| v_{0\alpha} \|_{H^1(\Omega_2)}
	\notag \\
	& \quad {}
	+
	\bigl\| j(u_{0\alpha}) \bigr\| _{L^1(\Omega_1)}^{1/2}
	+
	\bigl\| j(v_{0\alpha}) \bigr\| _{L^1(\Omega_2)}^{1/2} \Bigr),\label{est2}\\
	& \| \Delta u_\lambda \|_{L^2(0,T;L^2(\Omega_1))} 
	+ \kappa \| \Delta v_\lambda \|_{L^2(0,T;L^2(\Omega_2))}
	\notag \\
	& {} \le M_5 \Bigl( 1+ 
	\sqrt{\alpha} \| u_{0\alpha}-v_{0\alpha} \|_{L^2(S)} 
	+
	\| u_{0\alpha} \|_{H^1(\Omega_1)}
	+
	\kappa \| v_{0\alpha} \|_{H^1(\Omega_2)}
	\notag \\
	& \quad {}
	+
	\bigl\| j(u_{0\alpha}) \bigr\| _{L^1(\Omega_1)}^{1/2}
	+
	\bigl\| j(v_{0\alpha}) \bigr\| _{L^1(\Omega_2)}^{1/2} \Bigr),
	\label{est3}
\end{align}
\end{lemma}

\begin{proof}
Test \eqref{la1} by $\partial_t u_\lambda$, and use 
\eqref{la2}, \eqref{la3}. Similar to the equation 
of $v_\lambda$ by $\partial_t v_\lambda$, summing up the resultant we get  
\begin{align*}
	&
	\| \partial_t u_\lambda \|_{L^2(\Omega_1)}^2 
	+
	\| \partial_t  v_\lambda \|_{L^2(\Omega_2)}^2 
	+ \frac{1}{2} \frac{\d}{\dt} \int_{\Omega_1} | \nabla u_\lambda |^2 \dx  
	+ 
	\frac{\kappa}{2} \frac{\d}{\dt} \int_{\Omega_2} | \nabla v_\lambda |^2 \dx 
	 \\
	& \quad {} 
	+ \frac{\alpha}{2} \frac{\d}{\dt} 
	\| u_\lambda- v_\lambda\|_{L^2(S)}^2
	+ \frac{\d}{\dt}\int_{\Omega_1} b_\lambda(u_\lambda) \dx + \frac{\d}{\dt}\int_{\Omega_2} b_\lambda(v_\lambda) \dx  \\
	& = \bigl( g_1 - \pi_1 (u_\lambda), \partial_t u_\lambda \bigr)_{L^2(\Omega_1)}
	+  \bigl( g_2 - \pi_2 (v_\lambda), \partial _t v_\lambda \bigr)_{L^2(\Omega_2)}
	\\
	& \le  \| g_1 \|_{L^2(\Omega_1)}^2
	+ L_1^2 \| u_\lambda \|_{L^2(\Omega_1)}^2 
	+ \frac{1}{2} \| \partial_t u_\lambda \|_{L^2(\Omega_1)}^2 
	+ \| g_2 \|_{L^2(\Omega_2)}^2 
	+ L_2^2 \| v_\lambda \|_{L^2(\Omega_2)}^2
	+ \frac{1}{2} \| \partial_t v_\lambda \|_{L^2(\Omega_2)}^2 
\end{align*}
a.e.\ on $(0,T)$. 
Integrating this over $[0,\tau]$ with respect to time variable and using \eqref{est1}, 
we get 
\begin{align*}
	& \frac{1}{2} \int_0^\tau \bigl\| \partial_t u_\lambda (t) \|_{L^2(\Omega_1)}^2 \dt
	+ \frac{1}{2} \int_0^\tau \bigl\| \partial_t v_\lambda (t) \|_{L^2(\Omega_2)}^2 \dt
	+ \frac{1}{2} \int_{\Omega_1} \bigl| \nabla u_\lambda (\tau) \bigr|^2 \dx   
	\notag \\
	& {}
	+ 
	\frac{\kappa}{2} \int_{\Omega_2} \bigl| \nabla v_\lambda(\tau) \bigr|^2 \dx
	+ \frac{\alpha}{2} \bigl\| u_\lambda(\tau)-v_\lambda(\tau) \bigr\|_{L^2(S)}^2 
	 + \int_{\Omega_1} j_\lambda \bigl(u_\lambda(\tau) \bigr) \dx 
 + \int_{\Omega_2} j_\lambda \bigl(v_\lambda(\tau) \bigr) \dx
	\notag \\
	& \le 
	\| u_{0\alpha} \|_{H^1(\Omega_1)}^2
	+\kappa \| v_{0\alpha} \|_{H^1(\Omega_2)}^2
	+ \frac{\alpha}{2} \| u_{0\alpha}-v_{0\alpha} \|_{L^2(S)}^2 
	+ \bigl\| j (u_{0\alpha}) \bigr\|_{L^1(\Omega_1)}
	+ \bigl\| j (v_{0\alpha}) \bigr\|_{L^1(\Omega_2)}
	\notag \\
	& \quad {}
	+ \|g_1 \|_{L^2(0,T;L^2(\Omega_1))}^2
	+ \|g_2 \|_{L^2(0,T;L^2(\Omega_2))}^2
	+ L_1^2 T \|u_\lambda\|_{L^\infty(0,T;L^2(\Omega_1))}^2 +L_2^2 T \|v_\lambda\|_{L^\infty(0,T;L^2(\Omega_2))}^2, 
\end{align*}
where we used the 
property of $b_\lambda$ again. 
Thus, we conclude \eqref{est2}. 
The last estimate is obtained by the comparison in equations with \eqref{est1b} and \eqref{est2}. 
\end{proof}
\bigskip

\paragraph{\bf Proof of Proposition 2.1}
Thanks to the uniform estimates obtained in Lemmas~\ref{1st} and \ref{2nd}, we see that 
there exist a subsequence $\{ \lambda_n \}$ with $\lambda_n \to 0$, and 
targets $u_\alpha$, $\xi_\alpha$, $v_\alpha$, and $\psi_\alpha$ which depend on $\alpha \in (0,+\infty)$ such that  
\begin{align*}
	u_{\lambda_n} \to u_\alpha & \quad {\rm weakly~star~in~}H^1 \bigl( 0,T;L^2(\Omega_1) \bigr) 
	\cap L^\infty \bigl( 0,T;H^1(\Omega_1) \bigr), \\
	\Delta u_{\lambda_n} \to \Delta u_\alpha, \quad \beta_{\lambda_n} (u_{\lambda_n}) \to \xi_\alpha & 
	\quad {\rm weakly~in~}L^2 \bigl( 0,T;L^2(\Omega_1) \bigr), 
	 \\
	v_{\lambda_n} \to v_\alpha & \quad {\rm weakly~star~in~}H^1 \bigl( 0,T;L^2(\Omega_2) \bigr) 
	\cap L^\infty \bigl( 0,T;H^1(\Omega_2) \bigr), \\
	\Delta v_{\lambda_n} \to \Delta v_\alpha, \quad \beta_{\lambda_n} (v_{\lambda_n}) \to \psi_\alpha & 
	\quad {\rm weakly~in~}L^2 \bigl( 0,T;L^2(\Omega_2) \bigr)
\end{align*}
as $n \to +\infty$. Moreover, we have the compact imbedding $H^1(\Omega_i) \hookrightarrow L^2(\Omega_i)$, so applying the {A}ubin compactness theory \cite[Section 8, Corollary 4]{Sim87}, we get the strong convergence 
\begin{equation*}
	u_{\lambda_n} \to u_\alpha \quad {\rm in~}C\bigl([ 0,T ];L^2(\Omega_1) \bigr), \quad 
	v_{\lambda_n} \to v_\alpha \quad {\rm in~}C\bigl([ 0,T ];L^2(\Omega_2) \bigr),
\end{equation*}
these imply $\xi_\alpha \in \beta(u_\alpha)$ a.e.\ in $Q_1$, $\psi_\alpha \in \beta(v_\alpha)$ a.e.\ in $Q_2$ from the demi-closedness of the maximal monotone graph $\beta$, and 
\begin{equation*}
	\pi_1(u_{\lambda_n} ) \to \pi_1(u_\alpha) \quad {\rm in~}C\bigl([ 0,T ];L^2(\Omega_1) \bigr),\quad 
	\pi_2(v_{\lambda_n} ) \to \pi_2(v_\alpha) \quad {\rm in~}C\bigl([ 0,T ];L^2(\Omega_2) \bigr)
\end{equation*}
as $n \to +\infty$ from the {L}ipschitz continuities of $\pi_1$ and $\pi_2$, respectively. 
Finally, we can easily conclude that $(u_\alpha, \xi_\alpha, v_\alpha, \psi_\alpha)$ satisfies \eqref{heat1a}--\eqref{ini2a} from Lemma~\ref{chara} with \eqref{ee}. 
\hfill $\Box$

\subsection{Asymptotic analysis corresponding to closed junction states}

This subsection discusses the asymptotics as $\alpha \to 0$. 
The target problem corresponds to the situation where the permeability is equal to $0$. 
Namely, the problem is split into two problem in $\Omega_1$ and $\Omega_2$, respectively. 
We will prove it rigorously. 
\smallskip

In this subsection, we additionally prepare the following assumptions:
\begin{enumerate}
	\item[(A4)] 
	there exist $u_0 \in H^1(\Omega_1)$, $v_0 \in H^1(\Omega_2)$ such that $j(u_0) \in L^1(\Omega_1)$, $j(v_0) \in L^1(\Omega_2)$, and
	\begin{gather*}
	u_{0\alpha} \to u_0 \quad {\rm in~}H^1(\Omega_1), \quad j(u_{0\alpha}) \to j(u_0) \quad {\rm in~}L^1(\Omega_1), \\
	v_{0\alpha} \to v_0 \quad {\rm in~}H^1(\Omega_2), \quad j(v_{0\alpha}) \to j(v_0) \quad {\rm in~}L^1(\Omega_2)
	\quad {\rm as~}\alpha \to 0;
	\end{gather*}
	\item[(A5)] with respect to the order 
	\begin{equation*}
	\| u_0 -u_{0\alpha} \|_{L^2(\Omega_1)} = O(\alpha^{1/2}), \quad 
	\| v_0 -v_{0\alpha} \|_{L^2(\Omega_2)} = O(\alpha^{1/2}) \quad {\rm as~} \alpha \to 0.
	\end{equation*}
\end{enumerate}
For example, we chose $\beta(r)=r^3$ and $j(r)=(1/4)r^4$ in the case of {A}llen--{C}ahn equation. Therefore, the part of compatibility conditions $j(u_0) \in L^1(\Omega_1)$ and convergence of $\{ j(u_{0\alpha}) \}$ in the assumption {\rm (A4)} automatically hold from 
the {S}obolev imbedding theorem if $d \le 4$, because $u_0 \in H^1(\Omega) \subset L^4(\Omega_1)$ with $u_{0\alpha} \to u_0$ in $H^1(\Omega_1)$.

\begin{theorem}\label{split} Under the assumptions {\rm (A0)--(A4)}, 
let $(u_\alpha, \xi_\alpha, v_\alpha, \psi_\alpha)$ be 
the solution of \eqref{heat1a}--\eqref{ini2a} obtained in {\rm Proposition~\ref{pro}}. Then, 
there exists unique quadruplet $(u,\xi, v, \psi)$ in the following classes
\begin{gather*}
	u \in H^1\bigl(0,T; L^2(\Omega_1) \bigr) \cap L^\infty\bigl(0,T; H^1(\Omega_1) \bigr),
	\quad \Delta u, \xi \in L^2\bigl(0,T; L^2(\Omega_1) \bigr), \\
	v \in H^1\bigl(0,T; L^2(\Omega_2) \bigr) \cap L^\infty\bigl(0,T; H^1(\Omega_2) \bigr),
	\quad \Delta v, \psi \in L^2\bigl(0,T; L^2(\Omega_2) \bigr)
\end{gather*}
such that 
\begin{align*}
	u_\alpha \to u & \quad {\it weakly~star~in~}H^1 \bigl( 0,T;L^2(\Omega_1) \bigr) 
	\cap L^\infty \bigl( 0,T;H^1(\Omega_1) \bigr), \\
	u_\alpha \to u & \quad {\it in~}C\bigl([ 0,T ];L^2(\Omega_1) \bigr), 
	\quad 
	\xi_\alpha \to \xi \quad {\it weakly~in~} L^2 \bigl( 0,T ;L^2(\Omega_1) \bigr), \\
	v_\alpha \to v & \quad {\it weakly~star~in~}H^1 \bigl( 0,T;L^2(\Omega_2) \bigr) 
	\cap L^\infty \bigl( 0,T;H^1(\Omega_2) \bigr), \\
	v_\alpha \to v & \quad {\it in~}C\bigl([ 0,T ];L^2(\Omega_2) \bigr),
	\quad 
	\psi_\alpha \to \psi \quad {\it weakly~in~} L^2 \bigl( 0,T ;L^2(\Omega_2) \bigr)
\end{align*}
as $\alpha \to 0$, and $(u,\xi)$ satisfies
\begin{align}
	\partial _t u - \Delta u + \xi + \pi_1(u) =g_1, \quad \xi \in \beta(u) & 
	\quad {\it a.e.~in~} Q_1, \label{tag1}\\
	\partial_{\boldsymbol{n}} u = 0 
	& \quad {\it a.e.~on~} \Sigma_1, \label{tag2}\\
	\partial_{\boldsymbol{\nu}} u = 0 
	& \quad {\it a.e.~on~} S_T, \label{tag2S}\\
	u(0) = u_{0} 
	& \quad {\it a.e.~in~} \Omega_1.\label{tag3}
\end{align}
Therefore, if the subdomain $\Omega_1$ is enough smooth, for example $C^2$-class, then the additional regularity 
$u \in L^2(0,T;H^2(\Omega_1))$ is obtained. 
Analogously, $(v,\psi)$ satisfies  
\begin{align}
	\partial _t v - \kappa \Delta v +\psi+ \pi_2(v) =g_2, \quad \psi \in \beta(v) & 
	\quad {\it a.e.~in~} Q_2, \label{tag4}\\
	\kappa \partial_{\boldsymbol{n}} v = 0 
	& \quad {\it a.e.~on~} \Sigma_2, \label{tag5}\\
		\kappa \partial_{-\boldsymbol{\nu}} v = 0 
	& \quad {\it a.e.~on~} S_T, \label{tag5s}\\
	v(0) = v_{0} 
	& \quad {\it a.e.~in~} \Omega_2, \label{tag6}
\end{align}
and $v \in L^2(0,T;H^2(\Omega_2))$ if the subdomain $\Omega_2$ is smooth. 
Finally, under additional assumption {\rm (A5)}, 
the following rate of convergence is obtained:
\begin{equation*}
	\| u-u_\alpha \|_{C([ 0,T ];L^2(\Omega_1)) 
	\cap L^2 ( 0,T;H^1(\Omega_1)) }+
	\| v-v_\alpha \|_{C([ 0,T ];L^2(\Omega_2)) 
	\cap L^2 ( 0,T;H^1(\Omega_2)) } = O(\alpha^{1/2}) 
	\ {\it as~} \alpha \to 0.
\end{equation*}
\end{theorem}
\smallskip

Note that the rate of convergence is limited by the worse of the two rates: $\alpha^{1/2}$ and the convergence rate of the initial data. 
In order to discuss the limiting procedure, we use the uniform estimate for 
$u_\alpha$, $\xi_\alpha$, $v_\alpha$, and $\psi_\alpha$. 
Indeed, the estimates \eqref{est1}, \eqref{est1b}, and \eqref{est3} hold in the form where the left hand side is replaced from $u_\lambda, v_\lambda$
with $u_\alpha, v_\alpha$. 
The estimate \eqref{est2}, holds where $\beta_\lambda(u_\lambda)$ and $\beta_\lambda(v_\lambda)$ are replaced with $\xi_\alpha$ and $\psi_\alpha$, respectively. 
As a remark, some of uniform estimates also work to the case where $\alpha \to +\infty$ if we 
can clarify the dependence of $\alpha$. 
\smallskip

\paragraph{\bf Proof of Theorem 2.5.} 
Here, from the trace theory we see that there exist a positive constant $C_{\rm tr}>0$ such that 
$\| y \|_{L^2(S)} \le C_{\rm tr} \| y \|_{H^1(\Omega_i)}$ for all $y \in H^1(\Omega_i)$ $(i=1,2)$. Therefore, by using {\rm (A3)} and {\rm (A4)}
there exists $\alpha^* \in (0, +\infty)$ and 
a constant $M_5'>0$ such that 
$\| j (u_{0\alpha})\|_{L^1(\Omega_1)} \le M_5'$, $\| j (v_{0\alpha})\|_{L^1(\Omega_2)} \le M_5'$, 
and 
\begin{align*}
	\|u_{0\alpha}-v_{0\alpha}\|_{L^2(S)} 
	& \le C_{\rm tr} 
	\|u_{0\alpha}-u_{0} \|_{H^1(\Omega_1)} 
	+ \|u_{0}-v_{0}\|_{L^2(S)} 
	+  C_{\rm tr} \|v_{0}-v_{0\alpha}\|_{H^1(\Omega_2)} \\
	& \le M_5'
\end{align*}
for all $\alpha \in (0, \alpha^*]$. 
Using the uniform estimates obtained in Lemmas~\ref{1st} and \ref{2nd} we see that there exist a subsequence $\{ \alpha_n \}$ with $\alpha_n \to 0$, and targets $u$, $\xi$, $v$, and $\psi$ such that  
\begin{align*}
	u_{\alpha_n} \to u & \quad {\rm weakly~star~in~}H^1 \bigl( 0,T;L^2(\Omega_1) \bigr) 
	\cap L^\infty \bigl( 0,T;H^1(\Omega_1) \bigr), 
	\\
	\Delta u_{\alpha_n} \to \Delta u, \quad \xi_{\alpha_n} \to \xi & \quad {\rm weakly~in~}L^2 \bigl( 0,T;L^2(\Omega_1) \bigr),
	\\
	v_{\alpha_n} \to v & \quad {\rm weakly~star~in~}H^1 \bigl( 0,T;L^2(\Omega_2) \bigr) 
	\cap L^\infty \bigl( 0,T;H^1(\Omega_2) \bigr), 
	\\
	\Delta v_{\alpha_n} \to \Delta v, \quad \psi_{\alpha_n} \to \psi & \quad {\rm weakly~in~}L^2 \bigl( 0,T;L^2(\Omega_2) \bigr),
	\\
	\alpha_n (u_{\alpha_n}-v_{\alpha_n}) \to 0 & \quad {\rm in~}L^\infty \bigl( 0,T;L^2(S) \bigr)
\end{align*}
as $n \to +\infty$. Moreover, by applying the {A}ubin compactness theory again
\begin{equation*}
	u_{\alpha_n} \to u \quad {\rm in~}C\bigl([ 0,T ];L^2(\Omega_1) \bigr), \quad 
	v_{\alpha_n} \to v \quad {\rm in~}C\bigl([ 0,T ];L^2(\Omega_2) \bigr),
\end{equation*}
and these imply $\xi \in \beta(u)$ a.e.\ in $Q_1$, $\psi \in \beta(v)$ a.e.\ in $Q_2$, and 
\begin{equation*}
	\pi_1(u_{\alpha_n} ) \to \pi_1(u) \quad {\rm in~}C\bigl([ 0,T ];L^2(\Omega_1) \bigr),\quad 
	\pi_2(v_{\alpha_n} ) \to \pi_2(v) \quad {\rm in~}C\bigl([ 0,T ];L^2(\Omega_2) \bigr)
\end{equation*}
as $n \to +\infty$. Therefore, we see that $(u,\xi)$ satisfies \eqref{tag1}--\eqref{tag3} and 
$(v,\psi)$ satisfies \eqref{tag4}--\eqref{tag6}, respectively. 
From the assumption of {\rm (A2)}, it is easy to prove the uniqueness of problems 
 \eqref{tag1}--\eqref{tag3} and \eqref{tag4}--\eqref{tag6}, respectively. 
Thus, these subsequence convergence hold in the sense of all sequence. 
\smallskip
 
Finally we prove the rate of convergence. Take the difference between \eqref{tag1}, \eqref{tag2}, \eqref{tag2S}, \eqref{tag3} and \eqref{heat1a}, \eqref{jump1a}, \eqref{bc1a}, \eqref{ini1a}, respectively, to deduce
\begin{align}
	\partial _t (u-u_\alpha) - \Delta (u-u_\alpha) + (\xi -\xi_\alpha)+ \pi_1(u)-\pi_1(u_\alpha)  =0 &
	\quad {\rm a.e.~in~} Q_1, \label{def1}\\
	\partial_{\boldsymbol{\nu}} (u-u_\alpha)  = -\alpha (v_\alpha-u_\alpha) &
	\quad {\rm a.e.~on~} S_T, \label{def2}\\  
	\partial_{\boldsymbol{n}} (u-u_\alpha)  =0
	& \quad {\rm a.e.~on~} \Sigma_1, \label{def3}\\
	u(0)-u_\alpha(0)  = u_{0}-u_{0\alpha}
	& \quad {\rm a.e.~in~} \Omega_1.\label{def4}
\end{align}
Multiply \eqref{def1} by $u-u_\alpha$, integrate over $(0,\tau) \times \Omega_1$. We obtain
\begin{align}
	& \frac{1}{2} \bigl\| u(\tau)-u_\alpha(\tau) \bigr\|_{L^2(\Omega_1)}^2 
	+ 
	\int_0^\tau \! \! \int_{\Omega_1} \bigl| \nabla (u-u_\alpha) \bigr|^2 \dx \dt
	+ 
	\alpha \int_0^\tau \! \! \int_S (v_\alpha-u_\alpha)(u-u_\alpha) \d S \dt 
	\notag \\
	& \le L_1 \int_0^\tau \| u-u_\alpha \|_{L^2(\Omega_1)}^2 \dt 
	+ \frac{1}{2} \| u_0 - u_{0\alpha}\|_{L^2(\Omega_1)}^2 \label{u-ua}
\end{align}
for all $\tau \in [0,T]$ 
from \eqref{def2}--\eqref{def4}, {\rm (A1)}, and {\rm (A2)}. 
Analogously, we obtain the same kind of inequality for $v-v_\alpha$.  
Therefore, sum up it and \eqref{u-ua} we get 
\begin{align}
	& \frac{1}{2} \bigl\| u(\tau)-u_\alpha(\tau) \bigr\|_{L^2(\Omega_1)}^2 
	+\frac{1}{2} \bigl\| v(\tau)-v_\alpha(\tau) \bigr\|_{L^2(\Omega_2)}^2 
	+ 
	\int_0^\tau \! \! \int_{\Omega_1} \bigl| \nabla (u-u_\alpha) \bigr|^2 \dx \dt
	\notag \\
	& {}
	+ \kappa
	\int_0^\tau \! \! \int_{\Omega_2} \bigl| \nabla (v-v_\alpha) \bigr|^2 \dx \dt
	+ 
	\alpha \int_0^\tau \| v_\alpha-u_\alpha \|_{L^2(S)}^2 \dt 
	\notag \\
	& \le L_1 \int_0^\tau \| u-u_\alpha \|_{L^2(\Omega_1)}^2 \dt 
	+
	L_2 \int_0^\tau \| v-v_\alpha \|_{L^2(\Omega_2)}^2 \dt 
	+ 
	\frac{1}{2} \| u_0 - u_{0\alpha}\|_{L^2(\Omega_1)}^2
	\notag \\
	& \quad {}
	+ 
	\frac{1}{2} \| v_0 - v_{0\alpha}\|_{L^2(\Omega_2)}^2
	+ 
	\alpha \int_0^\tau \! \! \int_S (v_\alpha-u_\alpha)(v-u) \d S \dt 
	\notag \\
	& \le L_1 \int_0^\tau \| u-u_\alpha \|_{L^2(\Omega_1)}^2 \dt 
	+
	L_2 \int_0^\tau \| v-v_\alpha \|_{L^2(\Omega_2)}^2 \dt 
	+ 
	\frac{1}{2} \| u_0 - u_{0\alpha}\|_{L^2(\Omega_1)}^2
	\notag \\
	& \quad {}
	+ 
	\frac{1}{2} \| v_0 - v_{0\alpha}\|_{L^2(\Omega_2)}^2
	+ \frac{\alpha}{2} \int_0^\tau \| v_\alpha-u_\alpha \|_{L^2(S)}^2 \dt 
	+
	\frac{\alpha}{2} \int_0^\tau \| u-v \|_{L^2(S)}^2 \dt 
	\label{error}
\end{align}
for all $\tau \in [0,T]$. We can merge the fifth term of the right hand side to the left. 
Therefore, by virtue of the {G}ronwall inequality, we deduce that there exists a constant $M_2'>0$ such that   
\begin{align*}
	& \bigl\| u(\tau)-u_\alpha(\tau) \bigr\|_{L^2(\Omega_1)}^2 
	+\bigl\| v(\tau)-v_\alpha(\tau) \bigr\|_{L^2(\Omega_2)}^2 
	\notag \\ 
	& \le \bigl( \| u_0 - u_{0\alpha}\|_{L^2(\Omega_1)}^2
	+ \| v_0 - v_{0\alpha}\|_{L^2(\Omega_2)}^2
	+\alpha \bigr) M_2'
\end{align*}
for all $\tau \in [0,T]$, and \eqref{error} means the strong convergence 
in $L^2(0,T;H^1(\Omega_i))$. Finally, under the additional assumption {\rm (A5)}, we first 
obtain the rate of convergence with respect to $C([0,T];L^2(\Omega_i))$-norm and next, 
going back to \eqref{error} we get the rate of convergence with respect to $L^2(0,T;H^1(\Omega_i))$-norm. 
\hfill $\Box$. 

\subsection{Asymptotic analysis corresponding to single states}

Next, we discuss $\alpha \to + \infty$. 
As a summary of the previous subsection, we can discuss the asymptotic analysis 
$\alpha \to 0$, using the uniform estimates obtained in Lemmas~\ref{1st}, and \ref{2nd} again. Here, Lemma~\ref{1st} can be used also for the case where $\alpha \to +\infty$. 
However, in Lemma \ref{2nd} the term 
$\sqrt{ \alpha } \| u_{0\alpha}-v_{0\alpha} \|_{L^2(S)}$ 
should be treated delicately, when $\alpha \to +\infty$. 
\smallskip

In this subsection, we prepare the following assumptions:
\begin{enumerate}
	\item[(A6)] there exist $u_0 \in H^1(\Omega_1)$, $v_0 \in H^1(\Omega_2)$ such that $j(u_0) \in L^1(\Omega_1)$, $j(v_0) \in L^1(\Omega_2)$, and
	\begin{gather*}
	u_{0\alpha} \to u_0 \quad {\rm in~}H^1(\Omega_1), \quad j(u_{0\alpha}) \to j(u_0) \quad {\rm in~}L^1(\Omega_1), \\
	v_{0\alpha} \to v_0 \quad {\rm in~}H^1(\Omega_2), \quad j(v_{0\alpha}) \to j(v_0) \quad {\rm in~}L^1(\Omega_2)
	\quad {\rm as~}\alpha \to +\infty;
	\end{gather*}
	\item[(A7)] with respect to the order 
	\begin{equation*}
	\| u_{0\alpha}-v_{0\alpha} \|_{L^2(S)} = O(\alpha^{-1/2}) \quad {\rm as~} \alpha \to +\infty. 
	\end{equation*}
\end{enumerate}
Assumption {\rm (A6)} is very similar to {\rm (A4)}, 
but the difference is that $\alpha \to 0$ is replaced by $\alpha \to +\infty$. 
Moreover, the condition  
{\rm (A7)} under {\rm (A6)} implies that $u_0=v_0$ a.e.\ on $S$. 
\smallskip

Define subspaces of $H^1(\Omega_1)$ and $H^1(\Omega_2)$ by
\begin{gather*}
	V_{1,0}:=\bigl\{ y \in H^1(\Omega_1) : y = 0 \quad {\rm a.e.~on~} S \bigr\}, \quad 
	\| y \|_{V_{1,0}}:=\| y\|_{H^1(\Omega_1)} \quad {\rm for~} y \in V_{1,0},\\
	V_{2,0}:=\bigl\{ z \in H^1(\Omega_2) : z = 0 \quad {\rm a.e.~on~} S \bigr\}, \quad 
	\| z \|_{V_{2,0}}:=\| z\|_{H^1(\Omega_2)} \quad {\rm for~} z \in V_{2,0},
\end{gather*}
and denote the dual space of $V_{i,0}$ by $V_{i,0}'$ for $i=1,2$. 
Then, $V_{i,0} \hookrightarrow L^2(\Omega_i)$ and 
$H^1(\Omega_i) \hookrightarrow L^2(\Omega_i) \subset V_{i,0}'$ hold for $i=1,2$, respectively. 
As a remark, $V_{1,0}=H^1_0(\Omega_1)$ in the case of {\sc Figure} \ref{fig:case2}. 
One of the main theorem in this paper is as follows:

\begin{theorem}\label{cellf} Under the assumptions {\rm (A0)--(A3)} and {\rm (A6)}, 
let $(u_\alpha, \xi_\alpha, v_\alpha, \psi_\alpha)$ be 
the solution of \eqref{heat1a}--\eqref{ini2a} obtained in {\rm Proposition~\ref{pro}}. Then, 
there exist a subsequence $\{ \alpha_m \}$: $\alpha_m \to +\infty$ as $m \to +\infty$ and 
a quadruplet $(u,\xi, v, \psi)$ in the following classes
\begin{gather*}
	u \in H^1 \bigl(0,T;V_{1,0}' \bigr) \cap 	
	L^\infty \bigl(0,T; L^2(\Omega_1) \bigr)
	\cap L^2\bigl(0,T; H^1(\Omega_1) \bigr), \quad \xi \in L^2\bigl(0,T; L^2(\Omega_1) \bigr),\\
	v \in H^1\bigl(0,T; V_{2,0}' \bigr) 
	\cap L^\infty \bigl(0,T; L^2(\Omega_2) \bigr)
	\cap L^2\bigl(0,T; H^1(\Omega_2) \bigr), \quad \psi \in L^2\bigl(0,T; L^2(\Omega_2) \bigr)
\end{gather*}
such that 
\begin{align*}
	u_{\alpha_m } \to u & 
	\quad {\it weakly~star~in~}H^1 \bigl( 0,T;V_{1,0}'\bigr) 
	\cap L^\infty \bigl( 0,T;L^2(\Omega_1) \bigr) 
	\cap L^2\bigl(0,T; H^1(\Omega_1) \bigr), 
	\\
	u_{\alpha_m} \to u & \quad {\it in~} 
	C\bigl([ 0,T ];V_{1,0}' \bigr) \cap L^2 \bigl(0,T;L^2(\Omega_1) \bigr), 
	\\
	\xi_{\alpha_m} \to \xi & 
	\quad {\it weakly~in~} 
	L^2 \bigl(0,T;L^2(\Omega_1) \bigr), 
	\\
	v_{\alpha_m} \to v & \quad {\it weakly~star~in~}H^1 \bigl( 0,T;V_{2,0}' \bigr) 
	\cap L^\infty \bigl( 0,T;L^2(\Omega_2) \bigr)
	\cap L^2\bigl(0,T; H^1(\Omega_2) \bigr), \\
	v_{\alpha_m} \to v & \quad {\it in~}
	C\bigl([ 0,T ];V_{2,0}' \bigr) \cap L^2 \bigl(0,T;L^2(\Omega_2) \bigr), 
	\\
	\psi_{\alpha_m} \to \psi & 
	\quad {\it weakly~in~} 
	L^2 \bigl(0,T;L^2(\Omega_2) \bigr), 
	\\
	u_{\alpha_m}-v_{\alpha_m} \to 0 
	& 
	\quad {\it in~}L^2 \bigl( 0,T;L^2(S) \bigr)
\end{align*}
as $m \to +\infty$ and 
\begin{equation}
	u=v \quad {\it a.e.~on~}S.
	\label{uv}
\end{equation}
Moreover, $(u,\xi)$ satisfies
$u(0)=u_0$ a.e.\ in $\Omega_1$, $\xi \in \beta(u)$ a.e.\ in $Q_1$, and 
\begin{gather}
	\langle \partial _t u, y \rangle _{V_{1,0}',V_{1,0}} + 
	\int_{\Omega_1} \! \nabla u \cdot \nabla y \dx + \bigl(\xi+\pi_1(u),y \bigr)_{L^2(\Omega_1)} =
	(g_1,y)_{L^2(\Omega_1)}\label{weak1}
\end{gather}
for all $y \in V_{1,0}$ a.e.\ on $(0,T)$.
Analogously, $(v,\psi)$ satisfies  
$v(0)=v_0$ a.e.\ in $\Omega_2$, $\psi \in \beta(v)$ a.e.\ in $Q_2$, and
\begin{gather}
	\langle \partial _t v, z \rangle _{V_{2,0}',V_{2,0}} + 
	\kappa
	\int_{\Omega_2}\!  \nabla v \cdot \nabla z \dx
	+ \bigl(\psi+ \pi_2(v),z \bigr)_{L^2(\Omega_2)} =
	(g_2,z)_{L^2(\Omega_2)}\label{weak2}
\end{gather}
for all $z \in V_{2,0}$ a.e.\ on $(0,T)$. 
\end{theorem}

These weak formulations \eqref{weak1} and \eqref{weak2} are 
quite natural, because the homogeneous {N}eumann boundary condition on $\partial \Omega$ is hidden in both of them. Moreover, one of the transmission condition on $S$ is surely obtained by \eqref{uv}. In other word, these systems are connected on $S$ by \eqref{uv}. 
If we assume the strong condition {\rm (A7)} we obtain the following result, where 
the connection between two systems \eqref{weak1} and \eqref{weak2} becomes more clear 
in the following sense (see, the condition on $S_T$): 

\begin{theorem}\label{cellf2} Under the assumptions {\rm (A0)--(A3)} and {\rm (A6)}--{\rm (A7)}, 
let $(u_\alpha, \xi_\alpha,v_\alpha,\psi_\alpha)$ be 
the solution of \eqref{heat1a}--\eqref{ini2a} obtained in {\rm Proposition~\ref{pro}}. Then, 
there exists unique quadruplet $(u,\xi, v, \psi)$ in the following classes
\begin{gather*}
	u \in H^1\bigl(0,T; L^2(\Omega_1) \bigr) \cap L^\infty\bigl(0,T; H^1(\Omega_1) \bigr),
	\quad \Delta u, \xi \in L^2\bigl(0,T; L^2(\Omega_1) \bigr), 
	\\
	v \in H^1\bigl(0,T; L^2(\Omega_2) \bigr) \cap L^\infty\bigl(0,T; H^1(\Omega_2) \bigr),
	\quad \Delta v, \psi \in L^2\bigl(0,T; L^2(\Omega_2) \bigr)
\end{gather*}
such that 
\begin{align*}
	u_\alpha \to u & \quad {\it weakly~star~in~}H^1 \bigl( 0,T;L^2(\Omega_1) \bigr) 
	\cap L^\infty \bigl( 0,T;H^1(\Omega_1) \bigr), \\
	u_\alpha \to u & \quad {\it in~}C\bigl([ 0,T ];L^2(\Omega_1) \bigr), 
	\quad
	\xi_\alpha \to \xi \quad {\it weakly~in~}L^2 \bigl( 0,T;L^2(\Omega_1) \bigr), \\
	v_\alpha \to v & \quad {\it weakly~star~in~}H^1 \bigl( 0,T;L^2(\Omega_2) \bigr) 
	\cap L^\infty \bigl( 0,T;H^1(\Omega_2) \bigr), \\
	v_\alpha \to v & \quad {\it in~}C\bigl([ 0,T ];L^2(\Omega_2) \bigr), 
	\quad 
	\psi_\alpha \to \psi \quad {\it weakly~in~}L^2 \bigl( 0,T;L^2(\Omega_2) \bigr), \\
	u_\alpha-v_\alpha \to 0 
	& 
	\quad {\it in~}L^\infty \bigl( 0,T;L^2(S) \bigr)
\end{align*}
as $\alpha \to +\infty$. Moreover, $(u,\xi)$ satisfies
\begin{align}
	\partial _t u - \Delta u + \xi+\pi_1(u) =g_1, \quad \xi \in \beta(u) & 
	\quad {\it a.e.~in~} Q_1, \label{tag10}\\
	\partial_{\boldsymbol{n}} u = 0 
	& \quad {\it a.e.~on~} \Sigma_1,  \label{tag11}\\
	u=v, \quad 
	\bigl[ \partial_{\boldsymbol{\nu}} u \bigr]:= \partial_{\boldsymbol{\nu}} u-(-\kappa  \partial_{-\boldsymbol{\nu}} v) = 0 
	& \quad {\it a.e.~on~} S_T, \label{tag12}\\
	u(0) = u_{0} 
	& \quad {\it a.e.~in~} \Omega_1.\label{tag13}
\end{align}
Analogously, $(v,\psi)$ satisfies 
\begin{align}
	\partial _t v - \kappa \Delta v + \psi+\pi_2(v) =g_2, \quad \psi\in \beta(v) & 
	\quad {\it a.e.~in~} Q_2, \label{tag14}\\
	\kappa \partial_{\boldsymbol{n}} v = 0 
	& \quad {\it a.e.~on~} \Sigma_2, \label{tag15}\\
	v=u, \quad \bigl[-\kappa \partial_{-\boldsymbol{\nu}} v \bigr] := -\kappa \partial_{-\boldsymbol{\nu}} v - \partial_{\boldsymbol{\nu}} u  = 0 
	& \quad {\it a.e.~on~} S_T, \label{tag16}\\
	v(0) = v_{0} 
	& \quad {\it a.e.~in~} \Omega_2. \label{tag17}
\end{align}
Moreover, if $\partial _{\boldsymbol{\nu}} u \in L^2(0,T;L^2(S))$, for example the solution $u$ belongs to $L^2(0,T;H^2(\Omega_1))$, then 
under the following additional assumption
\begin{equation*}
	\| u_0 -u_{0\alpha} \|_{L^2(\Omega_1)} = O(\alpha^{-1/2}), \quad 
	\| v_0 -v_{0\alpha} \|_{L^2(\Omega_2)} = O(\alpha^{-1/2}) \quad {\it as~} \alpha \to +\infty,
\end{equation*}
the following rate of convergence is obtained: 
\begin{equation}
	\| u-u_\alpha \|_{C([ 0,T ];L^2(\Omega_1)) 
	\cap L^2 ( 0,T;H^1(\Omega_1)) }+
	\| v-v_\alpha \|_{C([ 0,T ];L^2(\Omega_2)) 
	\cap L^2 ( 0,T;H^1(\Omega_2)) }
	= O(\alpha^{-1/2}) \label{errorsp}
\end{equation}
as $\alpha \to + \infty$. 
\end{theorem}
\smallskip

As a remark, in the statement of 
Theorem~\ref{cellf2}, one of the transmission condition \eqref{tag12} can be obtained from 
\begin{equation}
	(\gamma _{\rm N} \nabla u)_{|_{S}}= - \kappa (\gamma _{\rm N} \nabla v)_{|_{S}-} \quad {\rm in~}H^{-1/2}(S) 
	\quad {\rm a.e.~on~} (0,T), \label{2ndtrans}
\end{equation}
namely the restriction $(\gamma _{\rm N} \nabla u)_{|_{S}} \in H^{-1/2}(S)$
of the trace $\gamma _{\rm N} \nabla u \in H^{-1/2}(\partial \Omega_1)$ 
is equal to the constant $-\kappa$ times 
the restriction $(\gamma _{\rm N} \nabla v)_{|_{S}-} \in H^{-1/2}(S)$
of $\gamma _{\rm N} \nabla v \in H^{-1/2}(\partial \Omega_2)$. 
Both of them are interpreted in $H^{-1/2}(S)$ and the jump $[\partial_{\boldsymbol{\nu}} u]$ is equal to $0$ it can be interpreted in $L^2(S)$. 
Therefore, it is written as $\partial_{\boldsymbol{\nu}} u -(-\kappa  \partial_{-\boldsymbol{\nu}} v)$ to emphasize that they cannot be separated into $\partial_{\boldsymbol{\nu}} u$ and 
$-\kappa  \partial_{-\boldsymbol{\nu}} v$ in $L^2$ sense. 
The condition \eqref{tag16} overlaps with \eqref{tag12}. 
However, it is intentionally added for the equation to $(v,\psi)$. 
\smallskip

For simplicity let $\kappa=1$, $\beta$ is singleton, and $\pi_1=\pi_2=:\pi$. Then, the corresponding solution $(u, v)$ can be interpreted as a solution of the following merged {A}llen--{C}ahn equation. Indeed, from the transmission condition \eqref{tag12} or \eqref{tag16}, we see that $\tilde{u}(t)+\tilde{v}(t) \in H^1(\Omega)$ for a.a.\ $t \in (0,T)$, that is,  
$U=\tilde{u} + \tilde{v} \in H^1(0,T;L^2(\Omega)) \cap L^\infty(0,T;H^1(\Omega))$, 
$\Delta U \in L^2(0,T;L^2(\Omega))$ satisfying 
\begin{align*}
	\partial _t U - \Delta U + B(U) =G:=\tilde{g}_1+\tilde{g}_2 & 
	\quad {\rm a.e.~in~} Q, \\
	\partial_{\boldsymbol{n}} U = 0 
	& \quad {\rm a.e.~on~} (0,T) \times \partial \Omega, \\
	U(0) = U_0:=\tilde{u}_{0}+\tilde{v}_0 
	& \quad {\rm a.e.~in~} \Omega,
\end{align*}
where $B(U)=\beta(U)+\pi(U)$. 
Therefore, if the domain $\Omega$ is enough smooth, 
for example $C^2$-class (see, the second figure of {\sc Figure} \ref{fig:case1}), then from the 
standard elliptic estimate 
we can gain $U \in L^2(0,T;H^2(\Omega))$. It corresponds to the single state any more. 
\smallskip

\begin{lemma} \label{3rd}
There exists a constant $M_6>0$ independent of $\alpha \in (0,+\infty)$ such that 
\begin{align}
	& \| \partial_t u_\alpha \|_{L^2(0,T;V_{1,0}')} 
	+ \| \partial _t v_\alpha \|_{L^2(0,T;V_{2,0}')} 
	\notag \\
	& \le 
	M_6 \Bigl( 1+\| u_{0\alpha}\|_{L^2(\Omega_1)}
	+  \| v_{0\alpha}\|_{L^2(\Omega_2)}
	+ \bigl\| j(u_{0\alpha}) \bigr\| _{L^1(\Omega_1)}^{1/2}
	+ 
	\bigl\| j(v_{0\alpha}) \bigr\| _{L^1(\Omega_2)}^{1/2}
	\Bigr).
	\label{time}
\end{align}
\end{lemma}

\begin{proof} 
Let $y \in V_{1,0}$, multiplying \eqref{heat1a} by $y$, 
integrating it over $\Omega_1$, 
using the integration by part with \eqref{bc1a}, then we get  
\begin{equation*}
	\bigl\langle \partial _t u_\alpha(t), y \bigr\rangle _{V_{1,0}', V_{1,0}}
	+ \int_{\Omega_1} \! \nabla u_\alpha \cdot \nabla y \dx
	= \bigl( g_1(t) - \xi_\alpha(t) - \pi_1\bigl( u_\alpha (t)\bigr), y \bigr)_{L^2(\Omega_1)}
\end{equation*}
for a.a.\ $t \in (0,T)$. 
Therefore,  
\begin{align*}
	\bigl\| \partial _t u_\alpha(t) \bigr\|_{V_{1,0}'} 
	& = \sup_{\scriptsize
	\substack{y \in V_{1,0} \\ \| y \|_{V_{1,0} } = 1} }
	\Bigl| \bigl\langle \partial _t u_\alpha(t), y \bigr\rangle_{V_{1,0}',V_{1,0}}
	\Bigr| \\
	& \le (1+L_1) \bigl\| u_\alpha(t) \bigr\|_{H^1(\Omega_1)} 
	+ 
	\bigl\| 
	g_1(t) 
	\bigr\|_{L^2(\Omega_1)} 
	+
	\bigl\| 
	\xi_\alpha(t) 
	\bigr\|_{L^2(\Omega_1)} 
\end{align*}
for a.a.\ $t \in (0,T)$. Thanks to \eqref{est1}--\eqref{est1b}, this means that there exists  
a constants $M_6>0$ independent of $\alpha \in (0,+\infty)$ such 
\begin{align*}
	\| \partial_t u_\alpha \|_{L^2(0,T;V_{1,0}')} 
	& \le M_6 \Bigl( 1+\| u_{0\alpha}\|_{L^2(\Omega_1)}
	+  \| v_{0\alpha}\|_{L^2(\Omega_2)}
	+ \bigl\| j(u_{0\alpha}) \bigr\| _{L^1(\Omega_1)}^{1/2}
	+ 
	\bigl\| j(v_{0\alpha}) \bigr\| _{L^1(\Omega_2)}^{1/2}
	\Bigr).
\end{align*}
It is true also for $\| \partial_t v_\alpha \|_{L^2(0,T;V_{2,0}')}$. 
Thus, we complete the proof of \eqref{time}. 
\end{proof}
\smallskip

\paragraph{\bf Proof of Theorem 2.6.} Lemmas~\ref{1st}, \ref{3rd}, and {\rm (A6)} imply that there exist a subsequence $\{ \alpha_m \}$ with $\alpha_m \to +\infty$, and targets $u$, $\xi$, $v$, and $\psi$ such that  
\begin{align}
	u_{\alpha_m} \to u & \quad {\rm weakly~star~in~}H^1 ( 0,T;V_{1,0}') 
	\cap L^\infty \bigl( 0,T;L^2(\Omega_1) \bigr)
	\cap L^2 \bigl( 0,T;H^1(\Omega_1) \bigr), 
	\notag \\
	\xi_{\alpha_m} \to \xi & \quad {\rm weakly~in~} L^2 \bigl( 0,T;L^2(\Omega_1) \bigr),
	\notag \\
	v_{\alpha_m} \to v & \quad {\rm weakly~star~in~}H^1 ( 0,T;V_{2,0}') 
	\cap L^\infty \bigl( 0,T;L^2(\Omega_2) \bigr)
	\cap L^2 \bigl( 0,T;H^1(\Omega_2) \bigr),
	\notag \\
	\psi_{\alpha_m} \to \psi & \quad {\rm weakly~in~} L^2 \bigl( 0,T;L^2(\Omega_2) \bigr),
	\notag \\
	u_{\alpha_m} - v_{\alpha_m} \to 0 & \quad {\rm in~}L^2\bigl( 0,T;L^2(S) \bigr) \label{strong}
\end{align}
as $m \to +\infty$. Moreover, thanks to $H^1(\Omega_i) \hookrightarrow L^2(\Omega_i) \subset V_{i,0}'$ hold for $i=1,2$, 
applying the compactness theory, we get
\begin{align*}
	u_{\alpha_m} \to u & \quad {\rm in~}C\bigl([ 0,T ];V_{1,0}' \bigr) \cap L^2 \bigl(0,T;L^2(\Omega_1) \bigr), \\
	v_{\alpha_m} \to v & \quad {\rm in~}C\bigl([ 0,T ];V_{2,0}' \bigr) \cap L^2 \bigl(0,T;L^2(\Omega_2) \bigr),
\end{align*}
these imply $\xi \in \beta (u)$ a.e.\ in $Q_1$, $\psi \in \beta (v)$ a.e.\ in $Q_2$, and 
\begin{equation*}
	\pi_1(u_{\alpha_m} ) \to \pi_1(u) \quad {\rm in~} L^2 \bigl(0,T;L^2(\Omega_1) \bigr), \quad 
	\pi_2(v_{\alpha_m} ) \to \pi_2(v) \quad {\rm in~} L^2 \bigl(0,T;L^2(\Omega_2) \bigr)
\end{equation*}
as $m \to +\infty$. Therefore, from {\rm (A6)} we see that for 
$u \in C([0,T];V_{1,0}') \cap L^\infty (0,T;L^2(\Omega_1))$, the function 
$u(0)$ makes sense and 
$u(0)=u_0$ a.e.\ in $\Omega_1$. 
Analogously, 
$v(0)=v_0$ a.e.\ in $\Omega_2$ holds. 
Moreover, from the weakly continuity of the trace operator from $H^1(\Omega_i)$ to $H^{1/2}(\partial \Omega_i) \subset L^2(\partial \Omega_i)$ with \eqref{strong}, 
one of the transmission condition \eqref{uv} holds from \eqref{strong}. 
Now, we have 
\begin{gather*}
	\langle \partial _t u_{\alpha_m}, y \rangle _{V_{1,0}',V_{1,0}} + 
	\int_{\Omega_1} \! \nabla u_{\alpha_m} \cdot \nabla y \dx + \bigl(\xi_{\alpha_m}+ \pi_1(u_{\alpha_m}),y \bigr)_{L^2(\Omega_1)} =
	(g_1,y)_{L^2(\Omega_1)}
\end{gather*}
for all $y \in V_{1,0}$ a.e.\ on $(0,T)$. Thus, taking $m \to +\infty$ in the above we deduce 
\eqref{weak1}. Analogously, we get \eqref{weak2}. \hfill $\Box$ \\
\smallskip

\paragraph{\bf Proof of Theorem 2.7.} The assumption {\rm (A6)} and the additional assumption {\rm (A7)} implies that 
\begin{equation*}
	\sqrt{\alpha} \| u_{0\alpha} - v_{0\alpha} \|_{L^2(S)}
	+
	\| u_{0\alpha} \|_{H^1(\Omega_1)} 
	+
	\kappa \| v_{0\alpha} \|_{H^1(\Omega_2)}
	+
	\bigl\| j(u_{0\alpha}) \bigr\| _{L^1(\Omega_1)}^{1/2}
	+
	\bigl\| j(v_{0\alpha}) \bigr\| _{L^1(\Omega_2)}^{1/2}
\end{equation*}
is uniformly bounded with respect to $\alpha \in (0,+\infty)$. 
Therefore, 
using Lemmas~\ref{1st} and \ref{2nd} we see that there exists a constant $M_7>0$ independent of $\alpha \in (0, +\infty)$ such that
\begin{gather*}
	\| \xi_\alpha \|_{L^2(0,T;L^2(\Omega_1))} 
	+ \| \psi_\alpha \|_{L^2(0,T;L^2(\Omega_2))}
	\le M_7, \notag \\
	\| u_\alpha \|_{H^1(0,T;L^2(\Omega_1))} 
	+ \| v_\alpha \|_{H^1(0,T;L^2(\Omega_1))} 
	+ \| u_\alpha \|_{L^\infty(0,T;H^1(\Omega_2))}
	+ \kappa \| v_\alpha \|_{L^\infty(0,T;H^1(\Omega_2))} \notag \\
	{} + \sqrt{\alpha} \| u_\alpha-v_\alpha \|_{L^\infty(0,T;L^2(S))} \le 
	M_7, \notag \\
	\| \Delta u_\alpha \|_{L^2(0,T;L^2(\Omega_1))} 
	+ \kappa \| \Delta v_\alpha \|_{L^2(0,T;L^2(\Omega_2))}
	\le M_7. \notag
\end{gather*}
Therefore, the same discussion in the proof of Theorem~\ref{split} works, that is, 
there exist a subsequence $\{ \alpha_m \}$ with $\alpha_m \to +\infty$, and targets 
$u$, $\xi$, $v$, and $\psi$ such that  
\begin{align*}
	u_{\alpha_m} \to u & \quad {\rm weakly~star~in~}H^1 \bigl( 0,T;L^2(\Omega_1) \bigr) 
	\cap L^\infty \bigl( 0,T;H^1(\Omega_1) \bigr), 
	\\
	\Delta u_{\alpha_m} \to \Delta u, \quad \xi_{\alpha_m} \to \xi  & \quad {\rm weakly~in~}L^2 \bigl( 0,T;L^2(\Omega_1) \bigr),
	\\
	v_{\alpha_m} \to v & \quad {\rm weakly~star~in~}H^1 \bigl( 0,T;L^2(\Omega_2) \bigr) 
	\cap L^\infty \bigl( 0,T;H^1(\Omega_2) \bigr), 
	\\
	\Delta v_{\alpha_m} \to \Delta v, \quad \psi_{\alpha_m} \to \psi & \quad {\rm weakly~in~}L^2 \bigl( 0,T;L^2(\Omega_2) \bigr),
	\\
	u_{\alpha_m} \to u & \quad {\rm in~}C\bigl([ 0,T ];L^2(\Omega_1) \bigr), \quad 
	v_{\alpha_m} \to v \quad {\rm in~}C\bigl([ 0,T ];L^2(\Omega_2) \bigr), 
	\\
	\pi_1(u_{\alpha_m} ) \to \pi_1(u)& \quad {\rm in~}C\bigl([ 0,T ];L^2(\Omega_1) \bigr), \quad 
	\pi_2(v_{\alpha_m} ) \to \pi_2(v) \quad {\rm in~}C\bigl([ 0,T ];L^2(\Omega_2) \bigr), 
	\\
	u_{\alpha_m}-v_{\alpha_m} \to 0 & \quad {\rm in~}L^\infty \bigl( 0,T;L^2(S) \bigr)
\end{align*}
as $m \to +\infty$. Moreover, we obtain $\xi \in \beta(u)$ a.e.\ in $Q_1$ and 
$\psi \in \beta (v)$ a.e.\ in $Q_2$, from the strong and weak convergence.  
Therefore, $u$ and $v$ satisfy initial conditions 
\eqref{tag13}, \eqref{tag17}, and one of the transmission condition \eqref{tag12}, that is, $u=v$ a.e.\ on $S$. 
We recall the equation \eqref{heat1a} for $u_{\alpha_m}$ and $\xi_{\alpha_m}$, taking 
the limit $m \to +\infty$ we recover the 
equation \eqref{tag10}. 
Analogously, we recover the equation \eqref{tag14} from \eqref{heat2a}. 
In order to recover the boundary conditions, 
we recall the following weak formulation:
\begin{equation}
	\bigl( \partial_t u_{\alpha_m} + \xi_{\alpha_m}+ \pi_1(u_{\alpha_m
}) -g_1, y \bigr)_{L^2(\Omega_1)}
	+ \int_{\Omega_1} \! \nabla u_{\alpha_m} \cdot \nabla y \dx  -
	\int_S \partial_{\boldsymbol{\nu}} u_{\alpha_m} y \dS =0
	\label{weak1a} 
\end{equation}
for all $y \in H^1(\Omega_1)$ and 
\begin{equation}
	\bigl( \partial_t v_{\alpha_m} +\psi_{\alpha_m}+ \pi_2(v_{\alpha_m}) -g_2,z \bigr) _{L^2(\Omega_2)} 
	+ \kappa \int_{\Omega_2} \!  \nabla v_{\alpha_m} \cdot \nabla z \dx -
	\kappa \int_S \partial_{-\boldsymbol{\nu}} v_{\alpha_m} z \dS =0
	\label{weak2a}
\end{equation}
for all $z \in H^1(\Omega_2)$, a.e.\ on $(0,T)$. 
Taking $w \in \tilde{H}^{1/2}(\Gamma_1)$, using the natural $0$-extension $\tilde{w}$ to $\partial \Omega_1$, 
we can take ${\mathcal R}_1 \tilde{w} \in H^{1}(\Omega_1)$ as the test function of the above \eqref{weak1a}, then 
\begin{equation*}
	\int_{\Omega_1} \bigl( \partial_t u_{\alpha_m} +\xi_{\alpha_m} + \pi_1(u_{\alpha_m})-g_1 \bigr){\mathcal R}_1 \tilde{w}  \dx 
	+ \int_{\Omega_1} \!  \nabla u_{\alpha_m} \cdot \nabla {\mathcal R}_1 \tilde{w} \dx=0,
\end{equation*}
where ${\mathcal R}_1: H^{1/2}(\partial \Omega_1) \to H^1(\Omega_1)$ is the recovery of the trace. 
So taking the limit $m \to +\infty$, using \eqref{Greenw}, \eqref{howto}, with 
the 
equation \eqref{tag10}, 
we get 
\begin{equation*}
	\bigl\langle (\gamma_{\rm N} \nabla u)_{|_{\Gamma_1}}, w \bigr \rangle_{H^{-1/2}(\Gamma_1),\tilde{H}^{1/2}(\Gamma_1)} = 
	\langle \gamma_{\rm N} \nabla u, \tilde{w} \rangle_{H^{-1/2}(\partial \Omega_1),H^{1/2}(\partial \Omega_1)} 
	= 0. 
\end{equation*}
It means $(\gamma_{\rm N} \nabla u)_{|_{\Gamma_1}}=0$ in $H^{-1/2}(\Gamma_1)$, namely from the comparison in the equation, we obtain 
$\partial_{\boldsymbol{n}}u =(\gamma_{\rm N} \nabla u)_{|_{\Gamma_1}}=0$ in $L^2(\Gamma_1)$. Thus, we deduce the 
part of the {N}eumann boundary condition \eqref{tag11}. Analogously, we deduce \eqref{tag15} for $v$.  
Next, we prove \eqref{2ndtrans}. 
For all $w \in \tilde{H}^{1/2}(S)$, firstly we consider the natural $0$-extension $\tilde{w}^1$ to $\partial \Omega_1$. Then, we can take ${\mathcal R}_1 \tilde{w}^1 \in H^{1}(\Omega_1)$ as the test function \eqref{weak1a}, then 
\begin{equation}
	\int_{\Omega_1} \bigl( \partial_t u_{\alpha_m} +\xi_{\alpha_m} + \pi_1(u_{\alpha_m})-g_1 \bigr){\mathcal R}_1 \tilde{w}^1  \dx 
	+ \int_{\Omega_1} \!  \nabla u_{\alpha_m} \cdot \nabla {\mathcal R}_1 \tilde{w}^1 \dx - \int_S \partial_{\boldsymbol{\nu}} u_{\alpha_m} w \dS =0. 
	\label{merge1}
\end{equation}
Secondly we consider the natural $0$-extension $\tilde{w}^2$ to $\partial \Omega_2$. 
Then, using the recovery ${\mathcal R}_2: H^{1/2}(\partial \Omega_2) \to H^1(\Omega_2)$ of the trace, 
we can take ${\mathcal R}_2\tilde{w}^2 \in H^{1}(\Omega_2)$ as the test function \eqref{weak2a}, then 
\begin{equation}
	\int_{\Omega_2} \bigl( \partial_t v_{\alpha_m} +\psi_{\alpha_m} + \pi_2(v_{\alpha_m})-g_2 \bigr){\mathcal R}_2 \tilde{w}^2  \dx 
	+ \kappa \int_{\Omega_2} \!  \nabla v_{\alpha_m} \cdot \nabla {\mathcal R}_2 \tilde{w}^2 \dx - \kappa \int_S \partial_{-\boldsymbol{\nu}} v_{\alpha_m} w \dS =0. 
	\label{merge2}
\end{equation}
Taking care of the hidden condition 
$\partial_{\boldsymbol{\nu}} u_{\alpha_m} = -\kappa \partial_{-\boldsymbol{\nu}} v_{\alpha_m}$ 
in \eqref{jump1a} and \eqref{jump2a}, we can merge \eqref{merge1} and \eqref{merge2}. 
Taking the limit $m \to +\infty$ to the resultant, using \eqref{Greenw} with 
the equations \eqref{tag10} and \eqref{tag14}, 
we get 
\begin{equation*}
	\langle \gamma_{\rm N} \nabla u, \tilde{w}^1 
	\rangle_{H^{-1/2}(\partial \Omega_1), H^{1/2}(\partial \Omega_1)}
	=
	-\kappa \langle \gamma_{\rm N} \nabla v, \tilde{w}^2 
	\rangle_{H^{-1/2}(\partial \Omega_2), H^{1/2}(\partial \Omega_2)},
\end{equation*}
that is, 
\begin{equation*}
	\bigl\langle (\gamma_{\rm N} \nabla u)_{|_{S}}, w \bigr \rangle_{H^{-1/2}(S),\tilde{H}^{1/2}(S)} = 
	-\kappa \bigl\langle (\gamma_{\rm N}
	 \nabla v)_{|_{S-}}, w \bigr \rangle_{H^{-1/2}(S),\tilde{H}^{1/2}(S)},
\end{equation*}
where we denote the restriction 
of $\gamma _{\rm N}
 \nabla v \in H^{-1/2}(\partial \Omega_2)$ 
to $\tilde{H}^{1/2}(S)$ by 
$(\gamma _{\rm N} \nabla v)_{|_{S}-} \in H^{-1/2}(S)$.  
Therefore, all equations \eqref{tag10}--\eqref{tag17} have been proven. 
\smallskip 

Next, we prove the uniqueness as the system of two problems. 
Going back to the weak formulations \eqref{weak1a} and \eqref{weak2a} for $\alpha_m$, 
let $y \in H^1(\Omega_1)$ 
and $z \in H^1(\Omega_2)$ satisfying $y=z$ a.e.\ on $S$, and 
taking them as test functions, respectively. 
Then, from \eqref{jump1a} and \eqref{jump2a} we have
\begin{equation*}
	\int_S \partial_{\boldsymbol{\nu}} u_{\alpha_m}y \dS = \int_S \alpha (v_{\alpha_m}-u_{\alpha_m}) y \dS  
	= -\kappa \int_S \partial_{-\boldsymbol{\nu}} v_{\alpha_m} z \dS.
\end{equation*}
By merging \eqref{weak1a} and \eqref{weak2a}, 
and by taking the limit $m \to +\infty$ 
\begin{align*}
	& \bigl( \partial_t u +\xi+\pi_1(u),y \bigr)_{L^2(\Omega_1)}
	+
	\bigl( \partial_t v +\psi + \pi_2(v),z \bigr)_{L^2(\Omega_2)} 
	+ \int_{\Omega_1} \!  \nabla u \cdot \nabla y \dx 
	\\
	& \quad {}+ \kappa \int_{\Omega_2} \! \nabla v \cdot \nabla z \dx =
	(g_1, y )_{L^2(\Omega_1)}
	+( g_2, z )_{L^2(\Omega_2)}.
\end{align*}
From now on, let 
$(u^{(i)}, \xi^{(i)}, v^{(i)}, \psi^{(i)})$ for $i=1,2$ be two solutions. 
Denote 
$\bar{u}:=u^{(1)}-u^{(2)}$ and $\bar{v}:=v^{(1)}-v^{(2)}$, respectively. 
From one of the transmission condition, $u^{(i)}=v^{(i)}$ a.e.\ on $S$, we see that 
$\bar{u}=\bar{v}$ a.e.\ on $S$. 
Take the difference of 
above weak formulations for 
$(u^{(1)}, v^{(2)})$ and $(u^{(2)}, v^{(2)})$, 
choose $y:=\bar{u}$ and $z := \bar{v}$, respectively. Then, by using the monotonicity of 
$\beta$ 
\begin{align*}
	\frac{1}{2} \frac{\d}{\dt} \| \bar{u} \|_{L^2(\Omega_1)}^2 
	+ 
	\frac{1}{2} \frac{\d}{\dt} \| \bar{v} \|_{L^2(\Omega_2)}^2 
	+
	\int_{\Omega_1} \!\!  | \nabla \bar{u} |^2 \dx  
	+ \kappa \int_{\Omega_2} \!\!  | \nabla \bar{v} |^2 \dx 
	\le L_1\| \bar{u} \|_{L^2(\Omega_1)}^2 +L_2 \| \bar{v} \|_{L^2(\Omega_2)}^2.
\end{align*}  
It means that the {G}ronwall inequality implies the uniqueness. 
Thus, these subsequence convergence hold in the sense of all sequence. 
\smallskip 

Finally, we prove the rate of convergence. Assume that $\partial _{\boldsymbol{\nu}} u \in L^2(0,T;L^2(S))$. 
Take the difference between equations
\eqref{tag10} and \eqref{heat1a}, boundary conditions 
\eqref{tag11} and \eqref{bc1a}, transmission conditions 
\eqref{tag12} and \eqref{jump1a}, and initial conditions 
\eqref{tag13} and \eqref{ini1a}, respectively.  
Multiplying the resultant by $u-u_\alpha$, integrating it over $(0,\tau) \times \Omega_1$, and using the 
differences of boundary and transmission conditions, and {\rm (A2)}, we deduce 
\begin{align}
	& \frac{1}{2} \bigl\| u(\tau)-u_\alpha(\tau) \bigr\|_{L^2(\Omega_1)}^2 
	+ 
	\int_0^\tau \! \! \int_{\Omega_1} \bigl| \nabla (u-u_\alpha) \bigr|^2 \dx \dt
	\notag \\
	& {}
	 - \int_0^\tau \langle \gamma_{\rm N} \nabla u, 
	 u-u_\alpha \rangle_{H^{-1/2}(\partial \Omega_1),H^{1/2}(\partial \Omega_1)} \dt 
	+ 
	\int_0^\tau \! \! \int_S \partial _{\boldsymbol{\nu}} u_\alpha (u-u_\alpha) \d S \dt 
	\notag \\
	& \le L_1 \int_0^\tau \| u-u_\alpha \|_{L^2(\Omega_1)}^2 \dt 
	+ \frac{1}{2} \| u_0 - u_{0\alpha}\|_{L^2(\Omega_1)}^2
	\label{same}
\end{align}
for all $\tau \in [0,T]$. 
Here, from $u=v$ a.e.\ on $S$, and \eqref{jump2a} we have 
\begin{align*}
	\int_S \partial _{\boldsymbol{\nu}} u_\alpha (u-u_\alpha) \d S
	& = \int_S (-\kappa \partial_{-\boldsymbol{\nu}} v_\alpha) (v-v_\alpha + v_\alpha
	- u_\alpha) \d S \\
	& = - \kappa  \int_S \partial_{-\boldsymbol{\nu}} v_\alpha (v-v_\alpha) \dS	
	 + \alpha \int_S |v_\alpha- u_\alpha|^2 \dS 
\end{align*}
a.e.\ in $(0,T)$. 
On the other hand from \eqref{tag11}, \eqref{tag13},
and the additional assumption $\partial _{\boldsymbol{\nu}} u \in L^2(0,T;L^2(S))$, we 
see that $-\kappa \partial _{-\boldsymbol{\nu}} v \in L^2(0,T;L^2(S))$ 
from \eqref{tag12}. Next, we have already known that 
$\partial _{\boldsymbol{n}} u =0 \in L^2(0,T;L^2(\Gamma_1))$ and  
$\kappa \partial _{\boldsymbol{n}} v =0 \in L^2(0,T;L^2(\Gamma_2))$. Therefore, using \eqref{tag11},  \eqref{tag12}, \eqref{tag15}, and 
 \eqref{tag16} we get
\begin{align*}
	& \langle \gamma_{\rm N} \nabla u, 
	 u-u_\alpha \rangle_{H^{-1/2}(\partial \Omega_1),H^{1/2}(\partial \Omega_1)} \\
	& = \int_{\Gamma_1} \partial_{\boldsymbol{n}} u (u - u_\alpha) \dg
	 + \int_{S} \partial_{\boldsymbol{\boldsymbol{\nu}}} u (u - u_\alpha) \dS	\\
	 	& = 
	  \int_{S} \partial_{\boldsymbol{\boldsymbol{\nu}}} 
	 u (v - v_\alpha + v_\alpha - u_\alpha) \dS \\
	 & = 
	 \int_{S} (-\kappa \partial_{\boldsymbol{\boldsymbol{-\nu}}} 
	 v) (v - v_\alpha) \dS
	 +
	 \int_{S} \partial_{\boldsymbol{\boldsymbol{\nu}}} 
	 u (v_\alpha - u_\alpha) \dS \\
	 & = -\kappa \langle \gamma_{\rm N} \nabla v, 
	 v-v_\alpha \rangle_{H^{-1/2}(\partial \Omega_2),H^{1/2}(\partial \Omega_2)}
	 +
	 \int_{S} \partial_{\boldsymbol{\boldsymbol{\nu}}} 
	 u (v_\alpha - u_\alpha) \dS 
\end{align*}
for all $\tau \in [0,T]$. 
Analogously, we obtain the same kind of inequality for $v-v_\alpha$ like
\eqref{same}. Therefore, merging them we deduce 
\begin{align}
	& \frac{1}{2} \bigl\| u(\tau)-u_\alpha(\tau) \bigr\|_{L^2(\Omega_1)}^2 
	+\frac{1}{2} \bigl\| v(\tau)-v_\alpha(\tau) \bigr\|_{L^2(\Omega_2)}^2 
	+ 
	\int_0^\tau \! \! \int_{\Omega_1} \bigl| \nabla (u-u_\alpha) \bigr|^2 \dx \dt
	\notag \\
	& {}
	+ \kappa
	\int_0^\tau \! \! \int_{\Omega_2} \bigl| \nabla (v-v_\alpha) \bigr|^2 \dx \dt
	+ 
	\alpha \int_0^\tau \| u_\alpha-v_\alpha \|_{L^2(S)}^2 \dt 
	\notag \\
	& \le L_1 \int_0^\tau \| u-u_\alpha \|_{L^2(\Omega_1)}^2 \dt 
	+
	L_2 \int_0^\tau \| v-v_\alpha \|_{L^2(\Omega_2)}^2 \dt 
	+ 
	\frac{1}{2} \| u_0 - u_{0\alpha}\|_{L^2(\Omega_1)}^2
	\notag \\
	& \quad {}
	+ 
	\frac{1}{2} \| v_0 - v_{0\alpha}\|_{L^2(\Omega_2)}^2
	+ 
	\frac{1}{2 \alpha}
	\int_0^\tau \! \! \| \partial _{\boldsymbol{\nu}} u \|_{L^2(S)}^2 \dt 
	+ 
	\frac{\alpha}{2} \int_0^\tau \| u_\alpha-v_\alpha \|_{L^2(S)}^2 \dt 
	\label{error2}
\end{align}
for all $\tau \in [0,T]$. Thus, applying the {G}ronwall inequality we get 
\begin{align}
	& \bigl\| u(\tau)-u_\alpha(\tau) \bigr\|_{L^2(\Omega_1)}^2 
	+\bigl\| v(\tau)-v_\alpha(\tau) \bigr\|_{L^2(\Omega_2)}^2 
	\notag \\
	& \le \left( \| u_0 - u_{0\alpha}\|_{L^2(\Omega_1)}^2
	+ 
	\| v_0 - v_{0\alpha}\|_{L^2(\Omega_2)}^2
	+ 
	\frac{1}{\alpha}
	\| \partial _{\boldsymbol{\nu}} u \|_{L^2(0,T;L^2(S))}^2 \right) 
	e^{2(L_1+L_2)T}
	\label{error3}
\end{align}
for all $\tau \in [0,T]$. Thus, \eqref{error2} and \eqref{error3} deduce the conclusion \eqref{errorsp}. 
\hfill $\Box$

\section{Relationship with the Mosco convergence}

Recalling the concept of convergence for the convex functional \cite{Att84, Sch00c}, we discuss the characterization of the 
asymptotic analysis which is discussed in the previous section. 
Base on the abstract theory of {A}ttouch \cite[Proposition~3.60, Theorem~3.66]{Att84} or \cite[Theorem~4.1]{Sch00c} we can expect that the convex functional $\varphi_\alpha$ which is defined by 
\eqref{phi_a} converges to some convex functional in a suitable sense. To 
clarify this fact, we recall the following concept of {M}osco convergence.

\begin{definition} Let $\phi_n$ and $\phi$ be proper, lower semi continuous, and convex functionals in a {H}ilbert space ${\mathcal H}$ for $n \in \mathbb{N}$. Then, it is said be 
$\phi_n$ converges to $\phi$ on ${\mathcal H}$ in the sense of {M}osco, if and only if 
the following conditions hold:
\begin{enumerate}
\item[(M1)] If the sequence $\{ u_n \}$ of ${\mathcal H}$ converges to $u$ weakly in ${\mathcal H}$, then the limit infimum inequality holds 
\begin{equation*}
	\phi(u) \le \liminf_{n \to +\infty} \phi_n (u_n);
\end{equation*} 
\item[(M2)] For all $u \in {\mathcal H}$, there exists a recovery sequence $\{u_n \}$ of ${\mathcal H}$ which converges to $u$ 
strongly in ${\mathcal H}$ such that
 \begin{equation*}
	\phi(u) \ge \limsup_{n \to +\infty} \phi_n (u_n).
\end{equation*}
\end{enumerate}
\end{definition}

Then we obtain the following characterization:
\begin{theorem}\label{Mosco} The convex functional $\varphi_\alpha$ converges to 
$\varphi_0$ in the sense of {M}osco as $\alpha \to 0$, and $\varphi_\alpha$ converges to 
$\varphi_\infty$ in the sense of {M}osco as $\alpha \to + \infty$ respectively, where
$\varphi_0, \varphi_\infty: {\mathcal H} \to [0,+\infty]$ are defined by 
\begin{gather}
	\varphi_0(U):=
	\begin{cases}
	\displaystyle 
	\frac{1}{2} \int_{\Omega_1} |\nabla u|^2 \dx + 
	\frac{\kappa}{2} \int_{\Omega_2} |\nabla v|^2 \dx
	& {\it if~} 
	U :=(u,v) 
	\in H^1(\Omega_1) \times H^1(\Omega_2), \\
	+\infty & {\it otherwise,}
	\end{cases}\label{phi_0}
\end{gather}
that is, $D(\varphi_0)=D(\varphi_\alpha)=H^1(\Omega_1) \times H^1(\Omega_2)$,  
\begin{gather}
	\varphi_\infty(U):=
	\begin{cases}
	\displaystyle 
	\frac{1}{2} \int_{\Omega_1} |\nabla u|^2 \dx + 
	\frac{\kappa}{2} \int_{\Omega_2} |\nabla v|^2 \dx
	& {\it if~} 
	U :=(u,v) 
	\in {\mathcal V}, \\
	+\infty & {\it otherwise,}
	\end{cases}
	\label{phi_inf}
\end{gather}
where ${\mathcal V}:=H^1(\Omega)$ and  
$D(\varphi_\infty)={\mathcal V}$.
\end{theorem} 

\begin{proof} Let $\alpha_n \to 0$ as $n \to +\infty$. Let $\{ U_n \}:=\{ (u_n,v_n)\}$ be 
a sequence of $\mathcal H$ which converges to $U:=(u,v)$ weakly in ${\mathcal H}$ as $n \to +\infty$.  
If $\liminf_{n \to +\infty} \varphi_{\alpha_n}(U_n)=+\infty$, then the condition (M1) is automatically holds. 
Therefore, assume that $r:=\liminf_{n \to +\infty} \varphi_{\alpha_n}(U_n)<+\infty$. Then there exists a 
subsequence $\{n_k \}$ such that $\lim_{k\to +\infty} \varphi_{\alpha_{n_k}}(U_{n_k})=r$ and 
\begin{align}
	u_{n_k} \to u & \quad {\rm weakly~in~} H^1(\Omega_1), \quad {\rm in~} L^2(\Omega_1), \label{m1}\\
	v_{n_k} \to v & \quad {\rm weakly~in~} H^1(\Omega_2), \quad {\rm in~} L^2(\Omega_2)\label{m2}
\end{align} 
as $k \to +\infty$. Therefore, from the positivity $(\alpha_{n_k}/2)\| u_{n_k}-v_{v_k}\|_{L^2(S)}^2$ and the 
weakly lower semicontinuity of the norm 
\begin{align*}
	\liminf_{n \to +\infty} \varphi_{\alpha_n}(U_n) = r 
	& = \lim_{k\to +\infty} \varphi_{\alpha_{n_k}}(U_{n_k}) \\
	& = \liminf_{k\to +\infty} \varphi_{\alpha_{n_k}}(U_{n_k})\\
	& \ge \varphi_{0}(U).
\end{align*}
Thus, we can show the condition {\rm (M1)}. 
The condition {\rm (M2)} clearly holds because the domains 
are same $D(\varphi_0)=D(\varphi_\alpha)=H^1(\Omega_1) \times H^1(\Omega_2)$. Therefore, 
for all $U \in D(\varphi_0)$, we can choose the recovery sequence
$U_n:=U$ itself and then 
$\limsup_{n \to +\infty} \varphi_{\alpha_n}(U_n) = \varphi_{0}(U)$ holds, because 
\begin{equation*}
	\frac{\alpha_{n}}{2} \int_S | u_n-v_n | ^2 \dS = \frac{\alpha_{n}}{2} \int_S | u-v | ^2 \dS\to 0
\end{equation*}
as $n \to \infty$, 
see the definition \eqref{phi_a} of $\varphi_\alpha$. We conclude that 
$\varphi_\alpha$ converges to 
$\varphi_0$ in the sense of {M}osco as $\alpha \to 0$.
\smallskip

Next, let $\alpha_m \to \infty$ as $m \to +\infty$. Let $\{ U_m \}$ be 
a sequence of $\mathcal H$ which converges to $U$ weakly in ${\mathcal H}$ as $m \to +\infty$.  
Assume that $r:=\liminf_{m \to +\infty} \varphi_{\alpha_m}(U_m)<+\infty$. Then, there exists a 
subsequence $\{m_k \}$ such that $\lim_{k\to +\infty} \varphi_{\alpha_{m_k}}(U_{m_k})=r$ and same kind 
of weak and strong convergence \eqref{m1} and \eqref{m2} hold for $\{u_{m_k}\}$ and $\{ v_{m_k}\}$
as $k \to +\infty$. Moreover, from the boundedness of $\{ \varphi_{\alpha_{m_k}}(U_{m_k})\}$, we see that 
there exists a positive constant $M_r >0$ such that 
\begin{equation}
	\frac{\alpha_{m_k}}{2} \int_S | u_{m_k}-v_{m_k} | ^2 \dS \le M_r \label{m3}
\end{equation}
for all $k \in \mathbb{N}$. 
The weakly convergence in $H^1(\Omega_1)$ and $H^1(\Omega_2)$ with \eqref{m3} guarantee the condition 
$u=v$ a.e.\ on $S$, that is, $U:=\tilde{u}+\tilde{v} \in {\mathcal V}=H^1(\Omega)$.   
Therefore, we can show the condition {\rm (M1)}. 
The condition {\rm (M2)} also holds. Indeed, we have $D(\varphi_\infty) \subset D(\varphi_\alpha)$. Thus, 
for all $U \in D(\varphi_\infty) ={\mathcal V}$, we can choose the recovery sequence 
$U_m:=U \in H^1(\Omega_1) \times H^1(\Omega_2)$ itself. Therefore, we deduce 
\begin{align*}
	\limsup_{m \to +\infty} \varphi_{\alpha_m}(U_m) & =  \limsup_{m \to +\infty} \varphi_{\alpha_m}(U) \\
	& =  \frac{1}{2} \int_{\Omega_1} |\nabla u|^2 \dx + 
	\frac{\kappa}{2} \int_{\Omega_2} |\nabla v|^2 \dx\\
	& = \varphi_{\infty}(U).
\end{align*}
Thus, we conclude that 
$\varphi_\alpha$ converges to 
$\varphi_\infty$ in the sense of {M}osco as $\alpha \to + \infty$.
\end{proof}
\smallskip

We can conclude this section that the boundary term in the convex functional $\varphi_\alpha$ works as the 
penalty term when $\alpha \to +\infty$. Moreover, 
the difference of the convex functionals $\varphi_0$ and $\varphi_\infty$ is only their domain 
(compare \eqref{phi_0} with \eqref{phi_inf}), 
However, the structure of the system is drastically changed.

\section{Non-autonomous permeability and blowing up situations}

In this section, we discuss the well-posedness for the case where 
$\alpha:=\alpha(t)$, i.e., when $\alpha$ depends on the time variable. 
Moreover, the situation $\alpha(t) \to +\infty$ as $t \to T^*<T$ is also considered. 
Hereafter, we will select one of the following assumptions:
\begin{enumerate}
	\item[(A8)] $\alpha \in W^{1,1}(0,T)$, \quad or 
	\item[(A9)] $\alpha \in C^1([0,T^*))$, $\alpha>0$ and $\alpha(t) \to +\infty$ as $t \to T^*$ where $T^* \in (0,T)$. 
\end{enumerate}
Hereafter, the initial conditions are independent of $\alpha$, that is, $u_{0\alpha}:=u_0$, 
$v_{0\alpha}:=v_0$. 

\begin{proposition}\label{pro2} Assume {\rm (A0)}--{\rm (A2)}, {\rm (A3)} with $u_{0\alpha}:=u_0$, 
$v_{0\alpha}:=v_0$, and {\rm (A8)} holds. 
Then, there exists a unique quadruplet $(u, \xi, v, \psi)$ of functions 
\begin{gather*}
	u \in H^1\bigl(0,T; L^2(\Omega_1) \bigr) \cap L^\infty\bigl(0,T; H^1(\Omega_1) \bigr),
	\quad \Delta u, \xi \in L^2\bigl(0,T; L^2(\Omega_1) \bigr), \\
	v \in H^1\bigl(0,T; L^2(\Omega_2) \bigr) \cap L^\infty\bigl(0,T; H^1(\Omega_2) \bigr),
	\quad \Delta v, \psi \in L^2\bigl(0,T; L^2(\Omega_2) \bigr)
\end{gather*}
such that \eqref{heat1a}--\eqref{bc2a}, 
$u(0)=u_0$ a.e.\ in $\Omega_1$, and $v(0)=v_0$ a.e.\ in $\Omega_2$ hold. 
\end{proposition}
\smallskip

For $t \in [0,T]$, define a proper, lower semi continuous, and convex functional 
$\varphi^t: {\mathcal H} \to [0,+\infty]$ by 
\begin{equation}
	\varphi^t(U):=
	\begin{cases}
	\displaystyle 
	\frac{1}{2} \int_{\Omega_1} |\nabla u|^2 \dx + 
	\frac{\kappa}{2} \int_{\Omega_2} |\nabla v|^2 \dx + 
	\frac{\alpha(t)}{2} \int_{S} |u-v|^2 \dS & {\rm if~} 
	U \in D(\varphi^t)
	, \\
	+\infty & {\rm otherwise,}
	\end{cases}\label{phi_t}
\end{equation}
where $D(\varphi^t):=H^1(\Omega_1) \times H^1(\Omega_2)$ which is independent of $t \in [0,T]$. 
In such a case, we can apply the abstract result of time dependent subdifferential operator 
\cite{AD75, Bar10}. Indeed for all $R>0$,  
$U:=(u,v) \in H^1(\Omega_1) \times H^1(\Omega_2)$ satisfying $\| U\|_{\mathcal H} \le R$, and $s,t \in [0,T]$
\begin{align*}
	\bigl| \varphi^t(U)-\varphi^s(U) \bigr| & \le 
	\bigl| \alpha(t) - \alpha(s) \bigr| \int_S |u-v|^2 \dS \\
	& \le \bigl| \alpha(t) - \alpha(s) \bigr| \bigl( \varphi^s(U) + C_R \bigr)
\end{align*}
where $C_R$ is a positive constant depends on $R$. Therefore, the assumption {\rm (A8)} is surely the 
condition that we can obtain the well-posedness for the evolution equation of the form \eqref{ee}
(see, e.g., \cite[p.54]{AD75}). Therefore, we omit the proof of Proposition~\ref{pro2}. 
\smallskip

The final part of this paper, 
we examine an interesting scenario in which the permeability coefficient 
exhibits blow up at a finite time $T^* \in (0,T)$, yet the system remains 
well-posed for times beyond $T^*$. Let $\kappa =1$, and $\pi_1=\pi_2=:\pi$.  

\begin{theorem}\label{scenario} Assume {\rm (A0)}--{\rm (A2)}, {\rm (A3)} with $u_{0\alpha}:=u_0$, 
$v_{0\alpha}:=v_0$, and {\rm (A9)} holds. 
Then, there exists a unique quadruplet $(u, \xi, v, \psi)$ of functions in the following class
\begin{gather}
	u \in H^1_{\rm loc}\bigl([0,T^*); L^2(\Omega_1) \bigr) 
	\cap L^\infty_{\rm loc}\bigl([0,T^*); H^1(\Omega_1) \bigr), \quad 
	\Delta u \in L^2_{\rm loc}\bigl([0,T^*); L^2(\Omega_1) \bigr), 
	\label{r1}
	\\
	u \in H^1(0,T^*; V_{1,0}') 
	\cap L^\infty \bigl(0,T^*; L^2(\Omega_1)\bigr)
	\cap L^2 \bigl(0,T^*; H^1(\Omega_1)\bigr),
	\label{r2}
	\\
	\xi \in L^2 \bigl( 0,T^*;L^2(\Omega_1) \bigr),
	\label{r7}
	\\
	v \in H^1 _{\rm loc} \bigl([0,T^*); L^2(\Omega_2) \bigr) 
	\cap L^\infty_{\rm loc} \bigl([0,T^*); H^1(\Omega_2) \bigr), \quad 
	\Delta v \in L^2_{\rm loc}\bigl([0,T^*); L^2(\Omega_2) \bigr),
	\label{r4}
	\\
	v \in 
	H^1(0,T^*; V_{2,0}' ) 
	\cap L^\infty \bigl(0,T^*; L^2(\Omega_2) \bigr)
	\cap L^2 \bigl(0,T; H^1(\Omega_2) \bigr)
	,\label{r5}
	\\
	\psi \in L^2 \bigl(0,T^*;L^2(\Omega_2) \bigr)
	\label{r8}
\end{gather}
such that 
\begin{align}
	\partial _t u- \Delta u + \xi + \pi(u) =g_1, \quad \xi \in \beta(u) & 
	\quad {\it a.e.~in~} (0,T^*) \times \Omega_1, \label{f1}\\
	\partial _t v - \Delta v + \psi + \pi(v) =g_2, \quad \psi \in \beta(v) & 
	\quad {\it a.e.~in~} (0,T^*) \times \Omega_2, \label{f2}\\
	\partial_{\boldsymbol{\nu}} u = \alpha(v-u) 
	& \quad {\it a.e.~on~} (0,T^*) \times S, \label{f3}\\
	 \partial_{-\boldsymbol{\nu}} v = \alpha(u-v)
	& \quad {\it a.e.~on~} (0,T^*) \times S,\label{f4}\\
	\partial_{\boldsymbol{n}} u = 0 
	& \quad {\it a.e.~on~} (0,T^*) \times \Gamma_1, \label{f5}\\
	\partial_{\boldsymbol{n}} v = 0 
	& \quad {\it a.e.~on~} (0,T^*) \times \Gamma_2, \label{f6}\\
	u(0) = u_{0} 
	& \quad {\it a.e.~in~} \Omega_1, \label{f7}\\
	v(0) = v_{0} 
	& \quad {\it a.e.~in~} \Omega_2. \label{f8}
\end{align}
Moreover, if the additional regularities 
$u \in C([0,T^*];H^1(\Omega_1)) \cap L^\infty(0,T^*;H^2(\Omega_1))$ and $v \in C([0,T^*];H^1(\Omega_2))$ hold and $\Omega$ is enough smooth, for example $C^2$-class, then $u(T^*)=v(T^*)$ a.e.\ on $S$ and 
$(u, v)$ can be extended beyond $T^*$ as the solution 
\begin{gather*}
	u \in C\bigl( [0,T];H^1(\Omega_1)\bigr), \quad 
	v \in C\bigl( [0,T];H^1(\Omega_2)\bigr),\\
	U := \tilde{u}+ \tilde{v}\in H^1 \bigl( T^*,T;L^2(\Omega) \bigr) \cap C \bigl( [T^*,T];H^1(\Omega) \bigr) 
	\cap L^2 \bigl( T^*,T;H^2(\Omega) \bigr),\\
	\Xi \in  L^2 \bigl( T^*,T;L^2(\Omega) \bigr)
\end{gather*}
of the homogeneous Neumann problem: 
\begin{align*}
	\partial_t U -\Delta U + \Xi + \pi(U) =G, \quad \Xi \in \beta(U) & \quad {\it a.e.~in~}(T^*,T) \times \Omega, \\
	\partial_{\boldsymbol{n}} U = 0 & \quad {\it a.e.~on~}(T^*,T) \times \partial \Omega, \\
	U(T^*)=\tilde{u}(T^*)+\tilde{v}(T^*)& \quad {\it a.e.~in~}\Omega
\end{align*}
corresponding to the case where $\alpha=+\infty$ in \eqref{f1}--\eqref{f8} a.e.\ on $(T^*,T)$. 
\end{theorem}

\begin{proof} For all $\bar{T} \in (0,T^*)$, 
we see from {\rm (A9)} that $\alpha \in C^{1}([0,\bar{T}]) \subset W^{1,1}(0,\bar{T})$. Therefore, applying Proposition \ref{pro2} on $[0,\bar{T}]$, we can see that there exists a unique quadruplet $(u,\xi,v,\psi)$ of functions 
\begin{gather*}
	u \in H^1\bigl(0,\bar{T}; L^2(\Omega_1) \bigr) \cap L^\infty\bigl(0,\bar{T}; H^1(\Omega_1) \bigr),
	\quad \Delta u, \xi \in L^2\bigl(0,\bar{T}; L^2(\Omega_1) \bigr), 
	\\
	v \in H^1\bigl(0,\bar{T}; L^2(\Omega_2) \bigr) \cap L^\infty\bigl(0,\bar{T}; H^1(\Omega_2) \bigr),
	\quad \Delta v, \psi \in L^2\bigl(0,\bar{T}; L^2(\Omega_2) \bigr)
\end{gather*}
such that \eqref{f1}--\eqref{f6} in the time interval $[0,\bar{T}]$, and 
initial conditions \eqref{f7}, \eqref{f8} hold. 
The arbitrariness of $\bar{T}$, we can gain the time local regularities
\eqref{r1} and \eqref{r4}. 
Moreover, we also get equations \eqref{f1}--\eqref{f6}. In this level, we do not know that $u(T^*)$ and $v(T^*)$ makes sense or not. Indeed, in Lemma~\eqref{2nd} the right hand side depends on $\alpha$. 
To begin with, since $\alpha$ is time dependent, the norm of $\alpha$ will appear on the right hand side. Thus, uniform estimates for the above classes cannot be expected. 
Therefore, going back to Lemma~\ref{3rd}. 
We can obtain the following estimate similarly to \eqref{time}
\begin{equation}
	\int_0 ^{t_n} \bigl\| \partial_t u(t) \bigr\|_{V_{1,0}'}^2 \dt  
	+ \int_0^{t_n} \bigl\| \partial _t v(t) \|_{V_{2,0}'}^2 \dt  
	\le 
	M_8
\end{equation}
for all $n \in \mathbb{N}$, 
where $M_8$ is a positive constant independent to $n \in \mathbb{N}$ and $\alpha$. 
Therefore, letting $n \to +\infty$ we obtain the regularities
$\partial _t u \in L^2(0,T^*; V_{1,0}')$ and $\partial_t v \in L^2(0,T^*; V_{2,0}')$, respectively.  
Next, multiply \eqref{f1} by $u$, add $\| u(t)\|_{L^2(\Omega_1)}^2$ to both sides, and 
integrate over $[0,t_n] \times \Omega_1$ where $t_n \nearrow T^*$ as $n \to +\infty$. 
Similarly, multiply \eqref{f2} by $v$, add $\|v(t)\|_{L^2(\Omega_2)}^2$ to both sides, and integrate over $[0,t_n] \times \Omega_2$. Then, summing up them and using the monotonicity of $\beta$ we deduce 
\begin{align}
	& \frac{1}{2}\bigl\| u(t_n) \bigr\|_{L^2(\Omega_1)}^2 
	+ \frac{1}{2}\bigl\|v(t_n) \bigr\|_{L^2(\Omega_2)}^2
	+ \int_0^{t_n} \bigl\| u (t) \bigr\|_{H^1(\Omega_1)}^2 \dt 
	+ \int_0^{t_n} \bigl\| v (t) \bigr\|_{H^1(\Omega_2)}^2 \dt
	\notag \\
	& {} + \int_0^{t_n} \alpha(t)\bigl\| u(t)-v(t)\bigr\|_{L^2(S)}^2 \dt \notag \\  
	& \le \frac{1}{2}\| u_0 \|_{L^2(\Omega_1)}^2 
	+ \frac{1}{2}\|v_0 \|_{L^2(\Omega_2)}^2
	+ \left( \frac{3}{2}+ L_1 \right)\int_0^{t_n} \bigl\| u(t) \bigr\|_{L^2(\Omega_1)}^2 \dt 
	\notag \\
	& \quad {}+ \left( \frac{3}{2}+ L_2 \right)\int_0^{t_n} \bigl\| v(t) \bigr\|_{L^2(\Omega_2)}^2 \dt 
	+ \frac{1}{2}\int_0^{t_n} \bigl\| g_1 (t) \bigr\|_{L^2(\Omega_1)}^2 \dt 
	+ \frac{1}{2}\int_0^{t_n} \bigl\| g_2 (t)\bigr\|_{L^2(\Omega_2)}^2 \dt.
\end{align}
Therefore, the {G}ronwall inequality implies that the similar estimate 
as in Lemma \ref{1st} holds, that is, there exists a constant $M_9>0$ such that
\begin{gather}
	\bigl\| u(t_n) \bigr\|_{L^2(\Omega_1)}^2 + \bigl\|v(t_n) \bigr\|_{L^2(\Omega_2)}^2 \le M_9, 
	\label{bound}\\
	\int_0^{t_n} \bigl\| u (t) \bigr\|_{H^1(\Omega_1)}^2 \dt 
	+ \int_0^{t_n} \bigl\| v (t)\bigr\|_{H^1(\Omega_2)}^2 \dt + 
	\int_0^{t_n} \alpha(t)\bigl\| u(t)-v(t)\bigr\|_{L^2(S)}^2 \dt \le M_9 \notag
\end{gather}
for all $n \in \mathbb{N}$, 
where $M_9$ is independent to $n \in \mathbb{N}$ and $\alpha:=\alpha(t)$. 
In order to obtain \eqref{r7}, 
we multiplying \eqref{f1} by $\xi$. 
Going back to the level of {Y}osida approximation to calculate 
$\xi$ against $\partial _t u$, we can also obtain 
the estimate similar to \eqref{btlam}
\begin{equation*}
	\int_0^{t_n} \bigl\| \xi(t) \bigr\| _{L^2(\Omega_1)}^2 \dt \le M_9.
\end{equation*} 
Letting $n \to +\infty$ we obtain the regularities 
$u \in H^1(0,T^*; V_{1,0}')\subset C([0,T^*]; V_{1,0}')$, $u \in L^2(0,T^*;H^1(\Omega_1))$, 
and $\xi \in L^2(0,T^*;L^2(\Omega_1))$. 
Therefore, $u(T^*)$ makes sense in $V_{1,0}'$. 
Now. from the boundedness \eqref{bound} we see that there exist a subsequence (not relabeled) and a target 
$\bar{u} \in L^2(\Omega_1)$ such that 
\begin{equation*}
	u(t_n) \to \bar{u} \quad {\rm weakly~in~} L^2(\Omega_1), \quad 
	u(t_n) \to u(T^*) \quad {\rm in~} V_{1,0}'
\end{equation*}
as $n \to +\infty$. By the uniqueness of the limit, the two coincide, that is, $u(T^*) \in L^2(\Omega_1)$. 
This implies \eqref{r2}. 
Analogously we can obtain $v(T^*) \in L^2(\Omega_2)$, \eqref{r5}, and \eqref{r8}. 
\smallskip 

Hereafter, we discuss the extension beyond $T^*$ as the solution of the homogeneous {N}eumann problem. 
Assume that $u \in C([0,T^*];H^1(\Omega_1)) \cap L^\infty(0,T^*;H^2(\Omega_1))$ and $v \in C([0,T^*];H^1(\Omega_2))$. 
These regularity conditions ensure that $u$ and $v$ possess traces $u=\gamma_0 u$, $\partial_{\boldsymbol{\nu}}u = \gamma_{\rm N}\nabla u$, and $v=\gamma_0 v$ on $S$ at $t=T^*$. 
Now recalling \eqref{f3}  to deduce 
\begin{align*}
	\bigl\| v (t_n)- u (t_n) \bigr\|_{L^2(S)} 
	 & = \frac{1}{\alpha(t_n)} 
	\bigl\| \partial _{\boldsymbol{\nu}} u (t_n) \bigr\|_{L^2(S)} \\
	& \le \frac{1}{\alpha(t_n)} C_{\rm tr} \|  u \|_{L^\infty(0,T^*;H^2(\Omega_1))}
\end{align*}
for all $n \in \mathbb{N}$, where $C_{\rm tr}>0$ is a positive constant independent of $n \in \mathbb{N}$. 
Indeed, there exist $C_{\rm tr}>0$ such that 
$\| \partial _{\boldsymbol{\nu}} z \|_{L^2(S)} \le C_{\rm tr} \| z \|_{H^2(\Omega_1)}$ for all $z \in H^2(\Omega_1)$. 
Therefore, letting $n \to +\infty$, we get $u(T^*)=v(T^*)$ a.e.\ on $S$. 
The argument that follows is actually simple: 
Taking $u(T^*)$ and $v(T^*)$ at $t=T^*$ as initial conditions, since $U(T^*):=\tilde{u}(T^*)+\tilde{v}(T^*) \in H^1(\Omega)$, a unique solution to the standard homogeneous {N}eumann problem exists, and we can extend the solution beyond $T^*$. Moreover, 
as seen from the discussion of Theorem~\ref{cellf2}, the solutions $u$ and $v$ are continuously connected across the blow-up time $t=T^*$ of $\alpha$. Namely, as the regularity from $t=0$ to $t=T$ we get 
$u \in C( [0,T];H^1(\Omega_1))$ and $v \in C( [0,T];H^1(\Omega_2))$. 
\end{proof}

\section{Conclusion}

In this paper, we discussed the well-posedness of {A}llen--{C}ahn type equations under {R}obin type transmission conditions in a form closely resembling the classical transmission problem, 
as well as the asymptotic analysis between problems and their rates of convergence. 
Although we considered two major cases for the domain settings (see, {\sc {F}igures} 1 and 2), 
{C}ase 2 requires less complicated function space settings compared to {C}ase 1, and thus we focused our discussion on {C}ase 1. 
Indeed, in the {C}ase 2 the triple junction doesn't appear. Therefore, $\tilde{H}^{1/2}(S)=H^{1/2}(S)=H^{1/2}(\partial \Omega_1)$. 
\smallskip

Based on the abstract theory of evolution equations governed by subdifferential operators, 
the case $0<\alpha< +\infty$ can be resolved by Proposition~\ref{pro}, 
while the arguments corresponding to $\alpha \to 0$
were discussed in {T}heorem~\ref{split}, and those corresponding to $\alpha \to +\infty$ 
were discussed in {T}heorems~\ref{cellf} and \ref{cellf2}. 
Regarding the gradient system structure underlying these arguments, 
Theorem~\ref{Mosco} provides a clear relationship with {M}osco convergence. 
This problem can also be discussed when the permeability 
$\alpha$ depends on time. 
In particular, based on {P}roposition~\ref{pro2}, even in the setting where 
$\alpha$ blows up at $T^*<T$ before the terminal time $T$, 
the system itself continues, and in Theorem~\ref{scenario} we discussed the solvability of the interesting problem where the discontinuous state at the contact surface $S$ 
between the two domains $\Omega_1$ and $\Omega_2$ before time $T^*$ becomes continuous beyond $T^*$ 
and the evolution continues while satisfying a different equation. 
These problems originate from \emph{gap junctions} that govern cell connections appearing in biology, 
and it can be said that they represent a drastic change from a transmission problem between two domains to a problem in a single domain $\Omega$, as if the cells assimilate through cell fusion in a sense when the permeability diverges. 
Although, the equation is a very simple parabolic equation, it is expected that extensions to more complex nonlinear parabolic equations and other equations are possible following this concept.
\smallskip
 
Finally, we conclude this paper by discussing in detail the relationship with the {C}ahn--{H}illiard equation under dynamic boundary conditions, which was not addressed in the previous sections. 
The results rigorously proved in this paper provide a clearer interpretation of the three types of the {GMS} model \cite{CF15, Gal06, GMS11}, the {LW} model \cite{CFW20, LW19}, and the {KLLM} model \cite{KLLM21} positioned in between. 
Let us first recall the boundary conditions appearing in the {KLLM} model. 
In the {C}ahn--{H}illiard equation under these dynamic boundary conditions, 
we consider the {C}ahn--{H}illiard equation in both the bulk $\Omega$ 
and on the surface $\partial \Omega$ as follows:
\begin{align}
	\partial _t u -\Delta \mu =0 
	& \quad {\rm in~} (0,T) \times \Omega, \label{bulk1}\\
	 \mu = -\Delta u + {\mathcal W}'(u) 
	 & \quad {\rm in~} (0,T) \times \Omega, \label{bulk2}\\
	\partial _t u +\partial _{\boldsymbol{n}} \mu -\Delta \theta =0 
	& \quad {\rm on~}(0,T) \times \partial \Omega, \label{surf1}\\
	\theta = \partial _{\boldsymbol{n}} u -\Delta_\Gamma u + {\mathcal W}'(u)
	& \quad {\rm on~}(0,T) \times \partial \Omega, \label{surf2}
\end{align}
where ${\mathcal W}'$ is the derivative of some double well potential, 
for example ${\mathcal W}'(u) =u^3-u$; 
the initial conditions are omitted here.
In other word, the system is a type of transmission problem between the bulk system 
\eqref{bulk1}--\eqref{bulk2} and the surface system \eqref{surf1}--\eqref{surf2}. 
Among the multiple unknown functions, the important ones are the bulk chemical potential $\mu$ 
and the surface chemical potential $\theta$, and the following {R}obin type boundary condition is set as the transmission condition connecting them:
\begin{equation*}
	L \partial _{\boldsymbol{n}} \mu = \theta - \mu
	\quad {\rm on~}(0,T) \times \partial \Omega,
\end{equation*}
(the parameter $L$ corresponds to $1/\alpha$ in this paper). 
If $L \to +\infty$ ($\alpha \to 0$), the above {R}obin type condition gives us 
\begin{equation*}
	\partial _{\boldsymbol{n}} \mu = 0
	\quad {\rm on~}(0,T) \times \partial \Omega,
\end{equation*}
and the system corresponds to {LW} model (in \eqref{surf1}, $\partial _{\boldsymbol{n}} \mu$ disappears).
This situation can be interpreted that the permeability constant $1/L$ converges to $0$, the equations 
\eqref{bulk1} and \eqref{surf1} are completely split. However, in this model these systems are indirectly connected by \eqref{surf2}. 
If $L \to 0$ ($\alpha \to +\infty$), the above condition gives us 
\begin{equation*}
	\theta = \mu
	\quad {\rm on~}(0,T) \times \partial \Omega,
\end{equation*}
and the system corresponds to {GMS} model.
This can be interpreted that the permeability converges to $+\infty$, 
the systems \eqref{bulk1}--\eqref{bulk2} and \eqref{surf1}--\eqref{surf2} are completely linked. 
These systems are directly connected and behave as if they were the same system. 
\smallskip

Such an interpretation has been justified by the results of this paper. 
Finally, as a future development, in the connection among the above three models, 
we were able to classify them by the movement of parameters through {R}obin type conditions related to the normal derivative appearing in one of the equations. 
Similarly, if we consider the same thing for another variable, namely $u$
appearing in \eqref{bulk1} and \eqref{surf1} of the above system, 
two models at opposite extremes would emerge, 
and considering both of them, 
models at two opposite extremes would appear again. 
Moreover, while the problem treated here is a transmission problem between the bulk and surface, 
there exists another line of research that considers the transmission problem between a main domain and a thin domain, and characterizes dynamic boundary conditions through the zero limit of the thickness of thin domain \cite{CR90, Lie13, GLR23, Miu25}. For example, in the Case~2, $\Omega_1$ is the main domain and 
the thin domain $\Omega_2$ disappears as the thickness-zero limit.  
Although this type of research is still developing, together with Attouch's results that lie at the beginning of this study, the asymptotic analysis of various problems is also an interesting research subject.
Further developments in this field related to these issues are also expected.


\end{document}